\documentclass[12pt]{amsart}
\usepackage{amsmath,amssymb,amsthm}
\usepackage{setspace}
\usepackage{moreverb}
\usepackage{mathrsfs}
\usepackage{url}
\usepackage[all]{xy}
\usepackage{lscape}
\usepackage{tikz}
\usetikzlibrary{arrows}
\usepackage{color}

\usepackage[utf8]{inputenc}
\usepackage[english]{babel}
\usepackage{amsmath, amsthm, amsfonts, amssymb}
\usepackage{graphicx}
\usepackage{color}
\usepackage{verbatim}
\usepackage[percent]{overpic}
\usepackage{geometry}
\usepackage{pdfpages}
\usepackage{enumitem}

\usepackage[makeroom]{cancel}

\newcommand{\ignore}[1]{}

\newtheorem{theorem}{Theorem}[section]
\newtheorem{lemma}[theorem]{Lemma}
\newtheorem{corollary}[theorem]{Corollary}
\newtheorem{proposition}[theorem]{Proposition}
\newtheorem{algorithm}[theorem]{Algorithm}
\theoremstyle{definition}
\newtheorem{definition}[theorem]{Definition}
\newtheorem{example}[theorem]{Example}

\newtheorem{cons}[theorem]{Construction}

\theoremstyle{remark}
\newtheorem{remark}[theorem]{Remark}

\numberwithin{equation}{section}

\input xy
\xyoption{all}

\newcommand{\bZ}{{\mathbb{Z}}}

\definecolor{grey}{rgb}{0.75,0.75,0.75}
\definecolor{orange}{rgb}{1.0,0.5,0.5}
\definecolor{brown}{rgb}{0.5,0.25,0.0}
\definecolor{pink}{rgb}{1.0,0.5,0.5}

\newcommand{\fM}{{\mathfrak m}}

\newcommand{\fa}{\mathfrak{a}}

\newcommand{\cC}{{\mathcal C}}

\newcommand{\cO}{{\mathcal O}}

\newcommand{\lra}{{\longrightarrow}}

\newcommand{\C}{\mathbb{C}}

\newcommand{\Oc}{\mathcal{O}}

\usepackage[colorlinks=true, linkcolor=blue, citecolor=green, urlcolor=cyan, pagebackref]{hyperref}

\begin{document}

\title[Monomial generators of complete planar ideals]{Monomial generators of complete planar ideals}

\author[M. Alberich-Carrami\~nana]{Maria Alberich-Carrami\~nana}

\author[J. \`Alvarez Montaner]{Josep \`Alvarez Montaner}

\author[G. Blanco]{Guillem Blanco }

\address{Departament de Matem\`atiques\\
Univ. Polit\`ecnica de Catalunya\\ Av. Diagonal 647, Barcelona
08028, Spain} \email{Maria.Alberich@upc.edu, Josep.Alvarez@upc.edu, Guillem.Blanco@upc.edu}

\thanks{ All three authors are supported by Spanish Ministerio de
Econom\'ia y Competitividad MTM2015-69135-P. MAC and JAM are also
supported by Generalitat de Catalunya 2014SGR-634 project and they are with the Barcelona Graduate School of Mathematics
(BGSMath). MAC is
also with the Institut de Rob\`otica i Inform\`atica Industrial
(CSIC-UPC)}



\begin{abstract}
We provide an algorithm that computes a set of generators for any complete ideal in a smooth complex surface.
More interestingly, these generators admit a presentation as monomials in a set of maximal contact elements associated to the minimal log-resolution of the ideal.
Furthermore, the monomial expression given by our method is an equisingularity invariant of the ideal.
As an outcome, we provide a geometric method to compute the integral closure of a planar ideal and we apply our algorithm to some families of complete ideals.
\end{abstract}

\maketitle

\section{Introduction}

Let $(X,O)$ be a germ of smooth complex surface and  $\cO_{X,O}$ the ring of germs of holomorphic functions in a
neighbourhood of $O$, which we identify with $\C\{x,y\}$ by taking local coordinates, and let $\fM$ be the maximal ideal
at $O$.  Let $\pi: X' \rightarrow X$ be a proper birational morphism that can be achieved as a sequence of blow-ups along 
a set of points. Given an effective $\mathbb{Z}$-divisor $D$ in $X'$  we may consider its associated (sheaf) ideal
$\pi_{\ast}\Oc_{X'}(-D)$ whose
stalk at $O$ we simply denote as $H_D$. This type of ideals were systematically studied by Zariski
in \cite{zariski-thm}. They are  complete ideals of $\Oc_{X,O}$ and  $\fM$-primary
whenever $D$ has exceptional support. Among the class of divisors defining the same complete ideal
we may find a unique maximal representative which happens to have the property of being \emph{antinef}.
Actually,  Zariski \cite{zariski-thm} showed that that above correspondence is in fact an isomorphism of semigroups between the set of complete $\fM$-primary
ideals and the set of antinef divisors with exceptional support.

\vskip 2mm

The aim of this work is to make computationally explicit this correspondence. Namely, given a proper birational morphism $\pi: X' \rightarrow X$
and an antinef divisor $D$ in $X'$ we provide an algorithm that gives us a system of generators of the ideal $H_D$.
Moreover, our method captures the topological type of $D$ and gives an output that discerns between the equisingularity class of $H_D$ and its analytic type: it produces a set of monomial expressions, which are invariant of the equisingularity class of $H_D$ (or equivalently invariant of the topological type of $D$); and when the set of variables in these monomial expressions are substituted by a set of maximal contact elements of the resolution $\pi $, they give a system of generators of $H_D$.
Our approach uses two main ingredients: the Zariski decomposition of complete ideals into simple ones, and subtle properties in the theory of adjacent ideals, which relay on results obtained in the study of sandwiched singularities and the Nash conjecture of arcs on these singularities.

\vskip 2mm

Our main result is an algorithm (see Algorithm \ref{alg:A1})  which we briefly describe.
We start by fixing a set $\{f_0,\dots,f_g\} $ of {\it maximal contact elements} associated to $\pi$. These are irreducible curves that are geometrically adapted to the proper birational morphism in the sense that they are parameterized by the dead-ends of the dual graph of the divisor $D$ (see Subsection \ref{max_contact} for details).
%
Once we have this fixed set we consider the decomposition $D=\rho_1 D_1 + \cdots +\rho_r D_r$ of the divisor $D$ corresponding to the Zariski decomposition $H_D= H_{D_1}^{\rho_1}\cdots H_{D_r}^{\rho_r}$, where the ideals $H_{D_i}$ are simple and complete (see \cite{zariski-thm}, \cite{casas} for details).  For each ideal $H_{D_i}$ we compute in a very specific unique way  an {\it adjacent ideal below} $H_{D_i}$, i.e. a codimension one ideal $H_{\widehat{D}_i}\varsubsetneq H_{D_i}$, and we prove in Proposition \ref{prop:good-curve} that we may pick a maximal contact element (or a power of it)  belonging to $H_{D_i}$ but not belonging to $H_{\widehat{D}_i}$. These
adjacent ideals are no longer simple  so we may repeat the same procedure. This process finishes after finitely many steps since these adjacent ideals have {\it smaller support}  (see Lemma \ref{lemma:smaller-support}). We visualize the steps of our algorithm as a tree of ideals with leaves corresponding to the maximal ideal. From the elements of maximal contact that we pick at each step  we  construct a system of generators of $H_{D}$.

\vskip 2mm

As an outcome of our method we obtain a set of generators of  a planar complete ideal given by monomials in this fixed set of {maximal contact elements} (see Theorem \ref{thm:main}). That is, we obtain an expression of the form $H_D = (\mathbf{f}^{\boldsymbol{\alpha}_1},\ldots ,  \mathbf{f}^{\boldsymbol{\alpha}_r})$, where $\mathbf{f}^{\boldsymbol{\alpha}}= f_0^{a_0}\cdots f_g^{a_g}$ for  $\boldsymbol{\alpha}=(a_0, \ldots , a_g) \in \bZ^{g+1}_{\geq 0}$. More interestingly, if we fix another set of maximal contact elements $\{f'_0,\dots,f'_g\} $, the algorithm returns the same monomial expression (in this new set) for the generators of  $H_D$ since it only uses the information of the equisingularity types of the maximal contact elements. Actually, the monomial expression that we have for $H_D$ also works for any other ideal $H_{D'} \subseteq \cO_{X,O}$ whose associated  divisor $D'$ has the same weighted dual graph as $D$. Hence the monomial expression given by our method is an equisingularity invariant of the ideal.

\vskip 2mm

As a first application of our algorithm we provide a method to compute the integral closure of any ideal
$\fa \subseteq \cO_{X,O}$. More precisely, given a log-resolution $\pi: X' \rightarrow X$ of the ideal $\fa$, let $D$ be the effective Cartier divisor such that $\fa\cdot\cO_{X'} = \cO_{X'}\left(-D\right)$. Then, the integral closure $\overline{\fa}$ is the ideal $H_D$.
We should point out that, given a set of generators of the ideal $\fa$ we may compute the divisor $D$ using the  algorithm
that we developed in \cite{ACAMB16}.

\vskip 2mm

There are some general algorithms to compute the integral closure as those proposed by
Vasconcelos \cite{Vas91,Vas00}, de Jong \cite{deJ98}, Leonard and Pelikaan \cite{LP03} or Singh and Swanson \cite{SS}
which have been implemented in computer algebra systems such as {\tt Macaulay2} \cite{GS} or {\tt Singular} \cite{DGPS}.
When comparing these algorithms with our method in the case of planar ideals we notice that the system of generators
that we produce is not minimal in general.
However, our method reflects more closely some geometric properties.
For instance, if two different ideals belong to the same
\emph{equisingularity class}, meaning that they have the similar weighted cluster of base points \cite[8.3]{casas} or equivalently their minimal log-resolution have the same weighted dual graph,  their integral closure have the same monomial expression in the respective set of maximal contact elements.

\vskip 2mm

The usefulness of our method becomes still more apparent when dealing with families of complete ideals dominated
by the same log-resolution. This is the case of  \emph{multiplier ideals} which is an invariant of singularities that has received
a lot of attention in recent years. Combining the algorithms given in \cite{ACAMB16} and \cite{ACAMDC}
with Algorithm \ref{alg:A1} we develop a method that, given a set of generators of a planar ideal,  returns a
set of generators of the corresponding multiplier ideals.
Since multiplier ideals are invariant up to integral closure, we obtain a result that resembles a formula given by Howald \cite{howald} in the sense
that the multiplier ideals of a monomial ideal (in the set of maximal contact elements)  is monomial as well.

\vskip 2mm

Another interesting family of complete ideals was considered by Teissier in \cite{Teissier}. These ideals are described by valuative conditions
given by the intersection multiplicity of the elements of $\cO_{X,O}$ with a fixed germ of plane curve.
With the help of Algorithm \ref{alg:A1} we may provide an explicit system of generators for these ideals.

\vskip 2mm

Finally, we would like to mention that Casas-Alvero \cite{Cas98} proposes a geometrical procedure to obtain a minimal set of generators of the complete ideal $H_D$  based on a construction of a filtration by complete ideals, namely a chain of consecutive adjacent ideals between $H_D$ and $\fM H_D$.
The main drawback in order to computationally implement this method is that one does not know a priori which log-resolution of $H_D$ will dominate all the log-resolutions of each complete ideal appearing in the filtration, since it depends on the choices of the adjacent ideal at each step. In contrast, our procedure fixes a proper birational morphism $\pi$ from the very beginning  and we always keep working with complete ideals whose log-resolution is dominated by $\pi$.
The further points which must be blown up after $\pi$ in the Casas-Alvero's method can not be chosen following an intrinsic unique way which works for any ideal whose log-resolution is dominated by $\pi$. Moreover these choices in Casas-Alvero's method determine the generators of $H_D$: the generators must go through certain infinitely near points and have to skip some others.  Both sets of points are among those blown up by $\pi$ and by the further blowing-ups.
Therefore an hypothetical computational implementation of Casas-Alvero's method would suffer from a lack of robustness and invariance, which are the strengths of our method. 
For discounted, the fact that the monomial expression given by our method is an equisingularity invariant of the ideal is not achieved by any other previous method.

\vskip 2mm

The algorithms developed in this paper have been implemented in the computer algebra system {\tt Magma} \cite{Magma} and are  available at

\begin{center}
 {\tt github.com/gblanco92/IntegralClosureDim2}
\end{center}

\vskip 2mm

\section{Preliminaries} \label{prelim}

The aim of this section is to introduce all the background that we will use throughout this work.

\subsection{Proper birational morphisms and infinitely near points}

Let $X$ be a smooth complex surface and $\cO_{X,O}\cong \C\{x,y\}$ the ring of germs of holomorphic functions in a neighborhood of a smooth point $O\in X$ and $\fM=\fM_{X,O} \subseteq \cO_{X,O}$ the maximal ideal at $O$. Given a proper ideal $\fa \subseteq \cO_{X,O}$ we have a decomposition  $\fa= (a) \cdot \fa'$, where $a\in \cO_{X,O}$  is the greatest common divisor of the elements of $\fa$ and  $\fa'$ is $\fM$-primary.

\vskip 2mm



Let \( \pi : X' \rightarrow X \) a proper birational morphism, with $X'$ a smooth surface. Any such proper birational morphism can be achieved as a sequence of blow-ups $${\displaystyle \pi:X'=X_{r+1}\xrightarrow{}X_{r}\xrightarrow{}\cdots \xrightarrow{}X_1=X}$$ with $X_{i+1} = \textrm{Bl}_{p_i}X_i$ for a point $p_i \in X_i$. Denote by \( E = \textrm{Exc}(\pi) \) the exceptional locus.
Since we are in a local framework, we shall assume throughout this work that $E$ is connected.
The set $K$ of points which have been blown up gives an index set for the exceptional components $\{E_p\}_{p\in K}$ in \( E \). We say that these points are a cluster of {\it infinitely near to the origin}  and we may establish a {\it proximity relation} between them. Namely, we say that a point $q\in K$ is \emph{proximate} to $p\in K$ if and only if $q$ belongs to the exceptional divisor $E_p$ corresponding to $p$ as proper or infinitely near point. We will denote this relation as $q\rightarrow p$ and we collect all these relations by means of the {\it proximity matrix} ${P}=(P_{p,q})$ defined as:
$$P_{p,q}:= \left\{
\begin{array}{rcl}
 1 & \hskip 1mm & \mbox{if } p=q, \\
-1 & \hskip 1mm & \mbox{if } \ q\rightarrow p, \\
 0 & \hskip 1mm & \mbox{otherwise.}
\end{array} \right.
$$
The proximity matrix is related to the negative-definite {\it intersection matrix } $N=(E_p\cdot E_q)$ by the formula $N=-P^\top P$. An infinitely near point $q$ is proximate to just one or two points. In the former case we say that $q$ is a \emph{free point}, and in the later it is a \emph{satellite point}. Besides, one can establish a \emph{partial ordering} in $K$. Namely, $q \le p$ if and only if $p$ is infinitely near to $q$.

A \emph{log-resolution} of a proper ideal $ \fa$ is a proper birational morphism $\pi: X' \rightarrow X$ where $X'$ is smooth such that there is an effective Cartier divisor $F$ satisfying $\fa \cdot \cO_{X'} = \cO_{X'}(-F)$, and the divisor $F + \textrm{Exc}(\pi )$ has simple normal crossings. The  \emph{weighted  dual graph} of $\pi$  is the dual graph of the divisor  $F =\sum_{1 \leq i \leq s} v_i(F) E_i$ weighted by the values $\{  v_i (F) \}_{1 \leq i \leq s}$, where $E_i= E_{p_i}$ with $1 \leq i \leq r$ are the exceptional components, and $E_i$ with $r<i \leq s$ are the non-exceptional irreducible components of $F$.

\begin{definition}\label{EquisingIdeal}
Two proper ideals $\fa $ and $\fa '$ are \emph{equisingular}, or belong to the same \emph{equisingularity class} if one  of the following equivalent conditions hold:
\begin{itemize}
  \item They have respectively minimal log-resolutions with equal weighted dual graph.
  \item The divisors $F$ and $F'$ associated to their respective minimal log-resolutions are homeomorphic (topologically equivalent).
  \item The weighted cluster of base points of $\fa $ and $\fa '$ are similar in the sense of \cite[8.3]{casas} in case $\fa $ and $\fa '$ are $\fM$-primary.
\end{itemize}
\end{definition}

Notice that when dealing with principal ideals $\fa =(f) $ and $\fa '=(g)$ the notion of equisingularity of the ideals is equivalent to the classical notion of equisingularity (or topological equivalence) of the germs $f=0$ and $g=0$.
Equisingular ideals have equisingular generic members. Moreover, equisingular $\fM$-primary complete ideals have besides equal codimension.

\subsection{Divisor basis}

Let  $\Lambda_\pi:= \bigoplus_{p\in K} \bZ E_p $ be the lattice of integral divisors in $X'$ with exceptional support. We have two different basis of this $\bZ$-module given by the {\it strict transforms} and the {\it total transforms} of the exceptional components. For simplicity we will also denote the strict transforms by $E_p$ and the total transforms by  $\overline{E_p}$. In particular, any divisor $D \in \Lambda_\pi$ can be presented in two different ways
$$D =\sum_{p\in K} v_p(D) E_p =\sum_{p\in K} e_p(D) \overline{E}_p,$$
where the weights $v_p(D)$ (resp. $e_p(D)$) are the {\it values} (resp. {\it multiplicities}) of $D$.
The relation between  values and multiplicities is given by the proximity relations
\begin{equation} \label{eq:proximity}
  v_{q}(D)= e_q(D) + \sum_{q \rightarrow p} v_{p}(D).
 \end{equation}
that provide a base change formula ${\bf e}^\top ={P} \cdot {\bf v}^\top$, where  we collect the multiplicities and values in the vectors ${\bf e}=(e_p(D))_{p\in K}$  and ${\bf v}=(v_p(D))_{p\in K}$, respectively.

\vskip 2mm

Aside from the total and strict transform basis $\{\overline{E_p}\}_{p\in K}$ and $\{E_p\}_{p\in K}$ of the lattice $\Lambda_\pi$ of exceptional divisors, we may also consider the {\it branch basis} $\{B_p\}_{p\in K}$ defined as the dual of $\{-E_p\}_{p\in K}$ with respect to the intersection form. Any divisor $D \in \Lambda_\pi$ has a presentation
\begin{equation}
 D=\sum_{p\in K} \rho_p(D) B_p,
\end{equation}
where $\rho_p(D) = -D \cdot E_p$ is the \emph{excess} at $p$ and the relation between excesses and multiplicities is given by ${\boldsymbol{\rho}}^\top = P^\top {\bf e}^\top$, where ${\boldsymbol{\rho}}=(\rho_p(D))_{p\in K}$  
denote the vectors of excesses.

\vskip 2mm

The \emph{support} \( \textrm{Supp}(D) \) of a divisor \( D \in \Lambda_{\pi} \) is the union of irreducible divisors of \( D \), i.e. if \( D = \sum_{\alpha} v_{\alpha}(D) E_{\alpha} \) is the expression of \( D \) in the strict transform basis then \( \textrm{Supp}(D) = \{E_{\alpha} \ |\ v_{\alpha}(D) \neq 0 \} \). We will also consider the same construction in the total transform basis.
Namely, let \( D = \sum_\alpha e_\alpha(D) \overline{E}_\alpha \) be the  expression of \( D \) in the total transform basis, then  we define
its \emph{support in the total transform basis} as \( \textrm{Supp}_{\overline{E}}(D) = \{\overline{E}_{\alpha} \ |\ e_\alpha(D) \neq 0\} \). To avoid confusion with the usual notion of support, we will always be explicit when referring to the support in the total transform basis.

\vskip 2mm

\begin{remark}
Let $\xi : f=0$ be the germ of curve defined by an element $f \in \mathcal{O}_{X, O}$ satisfying that $\pi$ dominates a resolution of the principal ideal $(f)$. The {\it total transform} of $\xi$ is the pull-back $\overline{\xi}:=\pi^*f$ and its {\it strict transform} $\xi'$ is the closure of $\pi^{-1}(\xi - \{O\})$. Then we have a presentation $$\overline{\xi}=\xi' + \sum_{p\in K} v_p(f) E_p = \xi' + \sum_{p\in K} e_p(f)\overline{ E_p} = \xi' + \sum_{p\in K} \rho_p(f) B_p ,$$ where $v_p(f):=v_p(D), e_p(f):=e_p(D)$ and $\rho_p(f):=\rho_p(D)$  for  $D = \textrm{Div}(\pi^*f)_{\rm exc}$.
\end{remark}

\vskip 2mm

Throughout this work we will be interested in germs of curves associated to branch basis divisors.

\begin{cons} \label{remark:curvettes}
Given a point $p \in K$, let  $f_p \in \mathcal{O}_{X, O}$ be an irreducible element such that its strict transform by the resolution $\pi$ intersects transversely $E_p$ at a smooth point of $E$. We have that $B_p = \textrm{Div}(\pi^*f_p)_{\rm exc}$. Conversely, any $g \in \mathcal{O}_{X, O}$ with $B_p = \textrm{Div}(\pi^*g)_{\textrm{exc}}$ is irreducible and its strict transform by the resolution $\pi$ intersects transversely $E_p$ at a smooth point of $E$.
\end{cons}

Since the elements $f_p \in \mathcal{O}_{X, O}$ are irreducible, the points $q \in K$ such that $e_q(B_p) \ne 0$ are totally ordered. Furthermore, the resolution of any $f_p \in \mathcal{O}_{X, O}, \hskip 2mm p\in K$ is dominated by $\pi$. Moreover, any $f$ whose resolution is dominated by $\pi$  can be written as a product $f= \prod_{p } f_p^{\rho_p}$ of suitable elements $f_p$ obtained as in Construction \ref{remark:curvettes}.


\vskip 2mm

\begin{remark} \label{remark:values_as_intersection}
Let $\xi : f=0$, $\eta : g=0$ be germs of curves such that $\pi$ dominates their resolutions and set $D = \textrm{Div}(\pi^*f)_{\rm exc}$ and $C = \textrm{Div}(\pi^*g)_{\rm exc}$. Noether's intersection formula \cite[Theorem 3.3.1]{casas} gives an expression for their intersection multiplicity at $O$ as $[\xi . \eta ]_O = \sum_p e_p(D) \cdot e_p(C) =(\sum_p e_p(D) \overline{E}_{p})(\sum_{p} e_p(C) \overline{E}_p) = D \cdot C $. Hence, $v_p(f)= [\xi . \eta _p ]_O = D \cdot B_p $, where $\eta _p : f_p=0$, with $f_p \in \mathcal{O}_{X, O}$ as in Construction \ref{remark:curvettes}.
\end{remark}

\vskip 2mm

\subsection{Complete ideals and antinef divisors} \label{complete-antinef}

Given an effective divisor $D = \sum_{p\in K} v_p E_p$ $\in \Lambda_\pi$,  we may consider its associated ideal sheaf
$\pi_* \mathcal{O}_{X'}(-D)$. Its stalk at $O$ is
\begin{equation} \label{eq:ideal}
H_D = \{ f \in \mathcal{O}_{X, O}\ |\ v_p(f) \ge v_p \ \textrm{for all}\ E_p \le D\},
\end{equation}
 This ideal $H_D$ is \emph{complete}, see \cite{zariski-thm}, and $\mathfrak{m}$-primary since $D$ has only exceptional support. Complete ideals are closed under all standard operations on ideals, except addition: the intersection, product and quotient of complete ideals is complete.

\vskip 2mm

Recall that an effective divisor $D \in \textrm{Div}(X')$ is called \emph{antinef} if $\rho_p = -D \cdot E_p \ge 0$, for every exceptional component $E_p, p \in K$. 
This notion
is equivalent, in the total transform basis, to
\begin{equation} \label{eq:proximity-inequality}
e_p(D) \ge \sum_{q \rightarrow p} e_q(D), \quad \textrm{for all}\ p \in K.
\end{equation}
These are usually called \emph{proximity inequalities}, see \cite[\S4.2]{casas}. 
By means of the relation given in Equation \eqref{eq:ideal}, Zariski \cite{zariski-thm} establishes an isomorphism of semigroups between the set of ideals \( H_D \) and the set of antinef divisors in \( \Lambda_\pi \).

\vskip 2mm

Given a non-antinef divisor $D$, one can compute an equivalent antinef divisor $\widetilde{D}$, called the \emph{antinef closure}, under the equivalence relation that both divisors define the same ideal, i.e. $\pi_*\mathcal{O}_{X'}(-D) = \pi_*\mathcal{O}_{X'}(-\widetilde{D})$, and via the so called \emph{unloading} procedure. This is
an inductive procedure which was already
described in the work of Enriques \cite[IV.II.17]{EC15} (see \cite[\S 4.6]{casas} for more details). The version that we present here is the one considered in \cite{ACAMDC}.

\begin{algorithm}{\emph{(}Unloading procedure \cite{ACAMDC}\emph{)}}
\label{alg:unloading}
\vskip 1mm

\noindent {\tt Input:} A divisor $D = \sum d_p E_p \in \Lambda_\pi$.

\noindent {\tt Output:} The antinef closure $\widetilde{D}$ of $D$.

\vskip 1mm

{\textit{\textbf{Repeat:}}}

\begin{enumerate}
 \item [$\cdot$] Define $\Theta:= \{ E_p \in \Lambda_{\pi} \ |\ \rho_p = - D \cdot E_p <0 \}.$
   \vskip 0.5mm
 \item [$\cdot$] Let $n_p = \left \lceil \frac {\rho_p}{E_p^2} \right \rceil$ for each $E_p \in \Theta$. Notice that $(D + n_p E_p)\cdot E_p \le 0$.
 \item [$\cdot$] Define a new divisor as $\widetilde{D}= D + \sum_{E_p \in \Theta} n_p E_p$.
\end{enumerate}

{\textit{\textbf{Until}}} the resulting divisor $\widetilde{D}$ is antinef.
\end{algorithm}


\vskip 2mm

A fundamental result of Zariski \cite{zariski-thm} establishes the unique factorization of complete ideals into simple complete ideals, an ideal being simple if it is not the product of ideals different from the unit.
A reinterpretation of this result in a more geometrical context is given by Casas-Alvero in \cite[\S 8.4]{casas}.


\begin{theorem}{\cite{zariski-thm}, \cite[\S 8.4]{casas}} \label{thm:zariski}
Let \( D \in \Lambda_\pi \) an antinef divisor expressed as \( D = \sum_{p \in K} \rho_p B_p \) in the branch basis. Then,
\begin{equation} \label{eq:factorization}
  H_D = \prod_{p \in K} H_{B_p}^{\rho_p},
\end{equation}
with $H_{B_p}$ being simple complete ideals for any $p \in K$.
\end{theorem}

In the sequel, we will call \emph{simple divisor} the unique antinef divisor defining a simple complete ideal. As a corollary of Theorem \ref{thm:zariski}, simple divisors in $\Lambda_\pi$ will always be equal to $B_p$ for some $p \in K$.

\vskip 2mm

\subsection{Maximal contact elements} \label{max_contact}
Let $\xi : f = 0$  be a germ of an irreducible element such that $\pi$ dominates the resolution of $(f)$. Its equisingularity class, or topological equivalence class, is determined by the set of {\it characteristic exponents} $\{m_1/n, \dots, m_g/n\}$, where $n = e_O(f)$ introduced in \cite{zariski-68}. Then, $f$ has a Puiseux series of the form
\begin{equation} \label{eq:puiseux}
s(x) = \sum_{\substack{j \in (n_0)\\ m_0 \leq j < m_1}} a_j x^{j/n}
+ \sum_{\substack{j \in (n_1)\\ m_1 \leq j < m_2}} a_j x^{j/n} +
\cdots + \sum_{\substack{j\in(n_{g-1}) \\ m_{g-1} \leq j < m_g}} a_j
x^{j/n} + \sum_{\substack{j \in (n_{g}) \\ j \ge m_g}}
a_j x^{j/n},
\end{equation}
where $m_0 = 0$, and $n_i = \gcd(n, m_1, \dots, m_i)$ so that, in particular, $n_0 = n$ and $n_g = 1$.

\vskip 2mm

Another way of describing the equisingularity class is using its {\it semigroup} (see \cite{zariski-86}). Consider $v_{\xi} : \mathcal{O}_{X, O} \rightarrow \mathbb{Z}$ the valuation induced by the intersection multiplicity of $\xi$ with another element $\zeta: h = 0, h \in \mathcal{O}_{X, O}, h \not\in (f)$. It is defined by $v_\xi(\zeta) = [\xi . \zeta]_O = \textrm{ord}_t\ h(t^n, s(t^n))$, and does not depend on the equation defining the germ $\zeta$ or the parameterization of $\xi$. Then, the semigroup of $\xi$ is
\[\Sigma(\xi) = \{ v_\xi(\zeta) \in \mathbb{Z}\ |\ \zeta : h = 0, h \in \mathcal{O}_{X, O},\ h \not\in (f) \},\]
and it is finitely generated. Namely,  $\Sigma(\xi) =  \langle n, \check{m}_1, \dots, \check{m}_{g} \rangle$ where
\begin{equation} \label{eq:semigroup-generators}
  \check{m}_i = \sum_{j = 1}^{i-1}\frac{(n_{j-1} - n_j)m_j}{n_{i-1}} + m_i, \quad \textrm{for}\ i = 1, \dots, g.
\end{equation}

Throughout this work we will be interested in those elements $f_i \in \mathcal{O}_{X, O}$ such that
$v_\xi(\gamma_i) = [\xi . \gamma_i]_O = \check{m}_{i}$, with $\gamma_i: f_i = 0$ for $i = 0, \dots, g$. They will be called \emph{maximal contact elements} of $\xi$. Several choices for each $f_i$ giving raise to different germs can be made; for instance, if $\xi$ is not tangent to the $y$-axis, then $f_0 = x$ and $f_1 = y + a_1x + a_2x^2+ \cdots$. In general, these elements can be explicitly computed. Namely, if the equation defining $\xi$ is in Weierstrass form, the maximal contact elements correspond to its approximate roots (see \cite{abhyankar1}, \cite{abhyankar2}). Alternatively, if one has a Puiseux series of $\xi$ as in \eqref{eq:puiseux}, then the equations of the maximal contact elements $f_i$ have Puiseux series:
\begin{equation} \label{eq:puiseux-generators}
s_i(x) = \sum_{\substack{j \in (n_0)\\ m_0 \leq j < m_1}} a_j x^{j/n} + \cdots + \sum_{\substack{j\in(n_{i-1}) \\ m_{i-1} \leq j < m_{i}}} a_j x^{j/n} + \cdots  \quad \textrm{for}\ i = 1, \dots g,
\end{equation}
where the non-explicit terms are assumed not to increase the polydromy order, \cite[\S 1.2]{casas}, $n/n_{i-1}$ of $s_i$, and either $f_0 = x$ or $f_0 = y$ depending on whether $\xi$ is tangent to the $x$-axis or the $y$-axis respectively, see \cite[\S 5.8]{casas}. Notice that the multiplicity at the origin of these maximal contact elements is $e_O(f_i) = n / n_{i-1}$ for $i = 1, \dots, g$.



\vskip 2mm

We introduce here the notion of \emph{maximal contact elements} for a proper birational morphism \( \pi : X' \longrightarrow X \). The maximal contact elements of \( \pi \) with exceptional divisor $E$ are those $f_p \in \mathcal{O}_{X, O}$ with $p \in K$ considered in Construction \ref{remark:curvettes}, such that the vertex $p$ is a dead-end of the dual graph of $E_{\textrm{exc}}$, i.e. the dual graph remains connected when the vertex is removed. In particular, a dead-end vertex will always be a free point or the origin $O$, since satellite points are always proximate to two points. A \emph{set of maximal contact elements} \( \{f_i\}_{i \in I} \) contains a unique $f_{p_i}=f_i$ for each dead-end vertex $p_i \in K$. Since these elements are determined by a finite number of valuative  conditions, the elements \( f_i \)  can always be chosen to be polynomials instead of power series. Moreover, these polynomials can be chosen in an intrinsic unique way which is valid for any $\pi $. This will be a key feature for
computational issues.

\vskip 2mm

This definition coincides with the one given for an irreducible element $f \in \mathcal{O}_{X,O}$.  The elements $f_p \in \mathcal{O}_{X, O}$ such that $p$ is a dead-end of the dual graph of the resolution of $(f)$ are exactly those with Puiseux series as in \eqref{eq:puiseux-generators}. Conversely, we can construct equations for the maximal contact elements $f_p$ with $p \in K$ of any divisor $F$ parameterized by $K$: the coordinates, as points in the projective line, $q \in E_{q'} \cong \mathbb{P}^1_{\mathbb{C}}, q \rightarrow q'$, for all $q \le p, e_q(f_p) \ne 0$ determine, and are determined by, the coefficients of a Puiseux series of any $f_p \in \mathcal{O}_{X, O}$, see \cite[\S 5.7]{casas}, which will be as in \eqref{eq:puiseux-generators}.

\begin{remark}
From any set of maximal contact elements \( \{f_i\} \) one can always recover the morphism \( \pi : X' \rightarrow X \) as the minimal resolution dominating the resolutions of \( \{f_i\} \).
\end{remark}

\begin{example} \label{ex:example1}
Consider the proper birational morphism \( \pi : X' \rightarrow X \) consisting in blowing-up four points with the configuration given by the following dual graph:

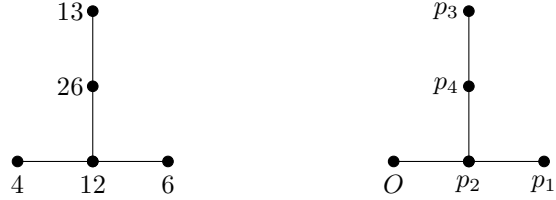
\tikzstyle{dual}=[circle, draw, fill=black!100, inner sep=0pt, minimum width=4pt]
\begin{center}
\begin{tikzpicture}[level/.style={sibling distance=120mm/#1}]
  \draw (-5,-2) node[dual] {} -- (-5,-1) node[dual] {} -- (-5,0) node[dual] {};
  \draw (-6,-2) node[dual] {} -- (-5,-2) node[dual] {} -- (-4,-2) node[dual] {};
  \draw (-5,-2.3) node {\footnotesize $p_2$};
  \draw (-6,-2.3) node {\footnotesize $O$};
  \draw (-4,-2.3) node {\footnotesize $p_1$};
  \draw (-5.3,-1) node {\footnotesize $p_4$};
  \draw (-5.3,0) node {\footnotesize $p_3$};

  \draw (-10,-2) node[dual] {} -- (-10,-1) node[dual] {} -- (-10,0) node[dual] {};
  \draw (-11,-2) node[dual] {} -- (-10,-2) node[dual] {} -- (-9,-2) node[dual] {};
  \draw (-10,-2.3) node {\footnotesize $12$};
  \draw (-11,-2.3) node {\footnotesize $4$};
  \draw (-9,-2.3) node {\footnotesize $6$};
  \draw (-10.3,-1) node {\footnotesize $26$};
  \draw (-10.3,0) node {\footnotesize $13$};
\end{tikzpicture}
\end{center}

Take, for instance, \( D = 4E_O + 6E_{p_1} + 12E_{p_2} + 13E_{p_3} + 26E_{p_4} \). The dead-end points are precisely $O, p_1, p_3$. Therefore, a set of maximal contact elements for \( \pi \) is $\{f_0, f_1, f_2\}$ with
\begin{align*}
  f_0 &= x + a_{2,0}x^2 + a_{0,2}y^2 + a_{1,1}xy + \cdots, \\
  f_1 &= y + b_2x^2 + b_3x^3 + \cdots, \\
  f_2 &= y^2 - x^3 + \sum_{3i + 2j > 6} c_{i,j} x^iy^j,
\end{align*}
and different choices of $a_{i,j}, b_i, c_{i, j} \in \mathbb{C}$ will give different sets of maximal contact elements. For instance, if all the coefficients are chosen to be zero, \( f_0 = x, f_1 = y, f_2 = y^2 - x^3 \), which are polynomials.
 \end{example}

\section{An algorithm to compute \texorpdfstring{\( H_D \)}{HD}} \label{algorithms}


Let $(X,O)$ be a germ of smooth complex surface and let \( \pi : X' \longrightarrow {X} \) be a proper birational transform.
In this section we present an algorithm which computes a set of generators for the \( \mathfrak{m} \)-primary ideal \( H_D \) for any divisor \( D \) with exceptional support in \( X' \), i.e. \( D \in \Lambda_\pi \).

\vskip 2mm

We briefly describe the main ideas behind Algorithm \ref{alg:A1}. We start with divisor \( D \in \Lambda_{\pi} \) which can be assumed to be antinef. It decomposes into simple divisors $D = \rho_{q_1} B_{q_1} + \dots + \rho_{q_r} B_{q_r}$ with all $\rho_{q_i} > 0$. For each simple divisor $B_{q_i}$, $q_i \neq O$, we compute the antinef closure of $B_{q_i} + E_{O}$ which we denote $\widehat{D}_i$. This new divisor describes a particular \emph{adjacent ideal} $H_{\widehat{D}_i}$ \emph{below} $H_{B_{q_i}}$, i.e. an ideal $H_{\widehat{D}_i} \varsubsetneq H_{B_{q_i}}$ such that $\dim H_{B_{q_i}}/H_{\widehat{D}_i} = 1$ as $\mathbb{C}$-vector space.
Next, we find, among the set of maximal contact elements of \( \pi \), an element $f \in \mathcal{O}_{X, O}$  belonging to $H_{B_{q_i}}$ but not to $H_{\widehat{D}_i}$. Now, $\widehat{D}_i$ is no longer simple but has smaller support than $B_{q_i}$ in the total transform basis. Therefore we may repeat the same procedure with $D := B_{q_j}, j < i$ until $D = B_O := \textrm{Div}(\pi^*\mathfrak{m})$.

\vskip 2mm

The first part of the algorithm can be represented by a tree where each vertex is an antinef divisor. The leaves of the tree are all $B_O=\textrm{Div}(\pi^*\mathfrak{m})$ and the root is the initial divisor \( D \). The second part of the algorithm traverses the tree bottom-up computing, in each node, the ideal associated to the divisor of that node. Using the notations from the above paragraph, given any node in the tree with divisor $D$, the ideal $H_D$ is computed multiplying the ideals in child nodes and adding the element $f$ to the resulting generators.

\vskip 2mm

Before giving a more explicit description of the algorithm, let us first state two technical results.
The first one presents some properties of adjacent ideals based on results obtained by Fern\'{a}ndez-S\'{a}nchez (see \cite{FS03}, \cite{FS05}, \cite{FS06}) in the study of sandwiched singularities and the Nash conjecture of arcs on these singularities.

\begin{proposition} \label{prop:adjacent-ideal}
Let $H_D$ be the complete ideal defined by an antinef divisor $D \in \Lambda_\pi$. Consider $\widehat{D}$ the antinef closure of $D + E_O$, obtained from $D + E_O$ by unloading on a given\footnote{$T$ is the set of points $p\in K$
that parameterize the support of $\widehat{D}-(D+E_O)$. Notice that \( T \) may be empty.} subset of  points $T \subseteq K$. Then, $H_{\widehat{D}} \varsubsetneq H_D$ are adjacent ideals if and only if $\rho_O(D) = 0$.

\vskip 2mm

Furthermore, if $H_{\widehat{D}} \varsubsetneq H_D$ are adjacent then,
\begin{enumerate}
  \item $\sum_{ p \in T} E_p$ is the connected component of $\sum_{ p \in K, \, \rho_p(D)=0} E_p$ containing $E_O$;
  \item $e_O(\widehat{D})=e_O(D)+1$, and $e_p(D) -1 \leq e_p(\widehat{D}) \leq e_p(D)$ for any $p \in K$, $p\neq O$; moreover $\rho_O(\widehat{D})>0$;
  \item if $p \in K \setminus T$ and $p$ is proximate to some point in $T$ then, $e_p(\widehat{D})=e_p(D) -1$.
\end{enumerate}
\end{proposition}

\begin{proof}
Consider the cluster $K'$ obtained from $K$ by adding $r:= \rho_O(D)+1$ free points $p_1, \ldots, p_r$ lying on $E_O$. Let $\pi' :Y' \rightarrow X$ be the composition of $\pi$ with the sequence of blow-ups of the points $p_1, \ldots, p_r$.
Denote by $\overline{G} \in \textrm{Div}(Y')$ the pullback of any $G \in \textrm{Div}(X')$.
For simplicity, denote the strict and the total transform basis by $\{E_p \}_{p\in K'}$ and $\{\overline{E_p}\}_{p\in K'}$ respectively in the lattice $\Lambda_{\pi'}$.

\vskip 2mm

Clearly, both $\overline{D}+ E_O$ and $\overline{D}+ E_{p_1}+ \cdots + E_{p_r}$ are not consistent, whereas $\overline{D}+ E_{p_1}+ \cdots + E_{p_i}$ are consistent for all $1\leq i <r$. Moreover, when applying the unloading procedure described in Algorithm \ref{alg:unloading}, we find that the antinef closures of $\overline{D}+ E_O$ and $\overline{D}+ E_{p_1}+ \cdots + E_{p_r}$ are the same, say it $\widehat{D'}$, and $e_{p_i}(\widehat{D'})=0$ for all $1 \leq i \leq r$. Indeed, the first step of the unloading procedure applied to  $\overline{D}+ E_O$ or $\overline{D}+ E_{p_1}+ \cdots + E_{p_r}$ gives the same divisor  $\overline{D}+ E_0+E_{p_1}+ \cdots + E_{p_r}$. Furthermore, $\widehat{D'}$ is the pullback of the antinef closure $\widehat{D}$ of $D + E_O$ in $\textrm{Div}(X')$ and hence they define the same complete ideal $H_{\widehat{D'}}=H_{\widehat{D}}$.

\vskip 2mm

Now, from \cite[\S 4.7]{casas}, the codimension of a complete ideal $H_G$ defined by a divisor $G \in \Lambda_\pi$ satisfies $$ \dim \mathcal{O}_{X, O}/H_{G} = \sum_{p \in K} \frac{e_p(\widetilde{G})(e_p(\widetilde{G})+1)}{2}
\leq \sum_{p \in K} \frac{e_p(G)(e_p(G)+1)}{2} \, ,
$$
where $\widetilde{G}$ is the antinef closure of $G$.
Hence, $$H_{\widehat{D}}= H_{\overline{D}+ E_{p_1}+ \cdots + E_{p_r}} \varsubsetneq H_{\overline{D}+ E_{p_1}+ \cdots + E_{p_{r-1}}} \varsubsetneq \ldots \varsubsetneq H_{\overline{D}+ E_{p_1}} \varsubsetneq H_D$$ is a chain of adjacent complete ideals, giving $\dim H_{B_{q_i}}/H_{\widehat{D}_i} = r= \rho_O(D)+1$. Therefore, $H_{\widehat{D}} \varsubsetneq H_D$ are adjacent, if and only if $\rho_O(D) = 0$.

\vskip 2mm

Finally, from \cite[2.1]{ACFS07} and \cite[4.6]{FS03} claim $i)$ follows. Claim $ii)$ and $iii)$ are consequences of \cite[4.2]{FS05} and \cite[2.2]{FS06} (see also \cite[3.1]{ACFS10}).
\end{proof}

\begin{remark}
Although there may be multiple adjacent ideals \( H_{\widehat{D}} \) to a fixed ideal \( H_D \), the adjacent ideal considered in Proposition \ref{prop:adjacent-ideal} is unique with the property that \( e_O(\widehat{D}) = e_O(D) + 1 \), so we will refer to it as \emph{the} adjacent ideal to \( H_D \). This property turns out to be crucial for the finiteness of the algorithm.
\end{remark}

\begin{remark}
Notice that if \( D \) is a simple divisor such that \( H_D \neq \mathfrak{m} \), then \( \rho_O(D) = 0 \) and \( H_{\widehat{D}} \) is always adjacent. Furthermore, the unloading step is always required, i.e. \( T \) is always non empty in this case.
\end{remark}

\begin{lemma} \label{lemma:nakayama}
Let $H_D$ be the complete ideal defined by an antinef divisor $D \in \Lambda_\pi$. The divisor $D + \overline{E}_O$ is antinef and $H_{D + \overline{E}_O} = \mathfrak{m} H_D$.
\end{lemma}
\begin{proof}
Clearly $\mathfrak{m} = \{f \in \mathcal{O}_{X, O}\ |\ e_O(f) \ge 1\}$. Thus, $\mathfrak{m} = H_{\overline{E}_O}$, and since $\overline{E}_O = B_O$ it is antinef and the result follows by the correspondence between antinef divisor and complete ideals in section \ref{complete-antinef}.
\end{proof}

\vskip 2mm


\vskip 2mm

\begin{algorithm} {(Generators for \( H_D \))} \label{alg:A1}

\vskip 2mm


\noindent {\tt Input:} A proper birational morphism \( \pi : X' \rightarrow X \) and an antinef divisor \( D \in \Lambda_{\pi} \).

\vskip 2mm


\noindent {\tt Output:} Generators for the ideal \( H_D \).

\vskip 2mm

\begin{enumerate}[label*=\arabic*.]


  \vskip 2mm

  \item[1.] Compute and fix a set of maximal contact elements $\{f_i\}$ with $i \in I$ of \( \pi \).

  \vskip 2mm

  \item[2.] Set $D^{(0)} := D$ and proceed from step $(0.1)$.

\vskip 2mm

  \hspace{-11mm} Step $(i)$:

\vskip 2mm

\begin{enumerate}

  \item[$i$.1] Decompose $D^{(i)}$ into $d_i := \#\{p \in K\ |\ \rho_p(D^{(i)}) > 0 \}$ simple divisors.

  \vskip 2mm

  \item[$i$.2] For each $j = 1, \dots, d_i$, consider $q_j \in \{p \in K\ |\ \rho_p(D^{(i)}) > 0\}$ and assume $B_{q_j} = \sum_{p \in K} e_p \overline{E}_p$.

  \vskip 2mm

  \begin{enumerate}

  \item[$i.j$.1]{\textbf{Stop at the maximal ideal:}}

  \vskip 2mm

  If $B_{q_j} = B_O := \textrm{\emph{Div}}(\pi^*\mathfrak{m})$, then set $H_{B_{q_j}} = (f_{i_0}, f_{i_1})$ for $i_0, i_1 \in I$ such that they are smooth and transverse at $O$, then stop. Otherwise, proceed from $i.j$.2.

\vskip 2mm

  \item[$i.j$.2]{\textbf{Compute the adjacent ideal to $H_{B_{q_j}}$:}}

\vskip 1mm

  Perform unloading on the divisor $B_{q_j} + E_O$ to get its antinef closure $\widehat{D}_j$.

\vskip 2mm

  \item[$i.j$.3]{\textbf{Select a maximal contact element in $H_{B_{q_j}} \setminus H_{\widehat{D}_j}:$}}

\vskip 2mm

  Let $p \in K$ be the last free point such that $e_p \neq 0$. Take $\tau_j \in I$ such that $e_p(f_{\tau_j}) = 1$ and $e_O(f_{\tau_j}) \le e_O$. Define the integer $n_{j} := e_O/e_O(f_{\tau_j})$.

\vskip 2mm

  \item[$i.j$.4]{\textbf{Recursive step}}:

\vskip 2mm

  Assume that $H_{\widehat{D}_j}$ has been computed after performing step $(i+1)$ with $D^{(i+1)} := \widehat{D}_j$.

\vskip 2mm

  \item[$i.j$.5]{\textbf{Set}}:
    \vskip -0.7cm
\[
  H_{B_{q_j}} = \left(f^{n_j}_{\tau_j}\right) + H_{\widehat{D}_j}.
\]
\end{enumerate}

\vskip 2mm

  \item[$i$.3]{\textbf{Apply Zariski's factorization theorem:}}

\vskip 2mm

  Compute the product $H_{D^{(i)}} = \prod_{j = 1}^{d_i} H_{B_{q_j}}^{\rho_{q_j}}$, giving generators $h_1, \dots, h_{s_i}$.

\vskip 2mm

  \item[$i$.4]{\textbf{Set, using Nakayama's lemma:}}
    \[
      \hspace{2cm}H_{D^{(i)}} = \Big(h_k \ \big|\ \pi^*h_k \not\in \mathcal{O}_{X'}\big(-D^{(i)} - \overline{E}_O\big),\ k = 1, \dots, s_i\Big)\mathcal{O}_{X, O}.
    \]

\end{enumerate}

\item[3.] \textbf{Return:} $H_D = H_{D^{(0)}}.$

\end{enumerate}

\end{algorithm}

\vskip 2mm


\begin{remark} \label{remark:nakayama}
In order to clarify some steps of the algorithm we point out the following:

\vskip 2mm
\begin{itemize}

\item[$\cdot$] At step 1 of Algorithm \ref{alg:A1} a set of maximal contact elements \( \{f_i\}_{i\in I} \) of \( \pi \) is fixed. The specific choice of the germs $f_i=0$ nor of the equations $f_i$ do not affect the output of the algorithm: the monomial expression remains the same for whatever choice, since the algorithm only uses the information of the equisingularity types of the maximal contact elements. However, as mentioned in Section \ref{max_contact}, the set \( \{f_i\} \) can be chosen in an intrinsic unique way, which is convenient for computational issues.


\vskip 2mm

\item[$\cdot$] At steps ($i.j.1$) and ($i.j.3$) of Algorithm \ref{alg:A1}  we have to choose maximal contact elements.  These choices are not necessary unique as several maximal contact elements may fulfill the required conditions. However, by simply keeping record of the dead-ends indexing these chosen maximal contact elements, our algorithm is able to produce the same output whatever the input equations of the ideal are, and whenever it is used with two different input ideals which have the same integral closure.
    
\vskip 2mm

\item[$\cdot$] Since the sheaf ideals $\mathcal{O}_{X'}(-D)$, with $D \in \textrm{Div}(X')$, are defined by valuations, testing whether the pullback of an element $f$  belongs to $\mathcal{O}_{X'}(-D)$ or not is only a matter of comparing the values $v_p(\textrm{Div}(\pi^*f))$ and $v_p(D)$ for all $p \in K$.

\vskip 2mm

\item[$\cdot$] It is clear from Nakayama's lemma and Lemma \ref{lemma:nakayama} that a set of elements of $\mathcal{O}_{X, O}$ is a system of generators of $H_{D}$ if and only if its classes modulo $H_{D + \overline{E}_O}$ are a system of generators of $H_{D}/H_{D + \overline{E}_O}$ as $\mathbb{C}$-vector space. Equivalently, any element of $H_{D + \overline{E}_O}$ is redundant in a system of generators of $H_D$.

\end{itemize}
\end{remark}

\begin{example} \label{ex:example}

We will compute \( H_D \) for the divisor \( D \) and the morphism \( \pi \) from Example \ref{ex:example1}. Let us fix the set of maximal contact elements \( f_0 = x, f_1 = y, f_2 = y^2 - x^3 \). The steps of Algorithm \ref{alg:A1} applied to \( D = 4E_O + 6E_{p_1} + 12E_{p_2} + 13E_{p_3} + 26E_{p_4} \) will be illustrated by means of the tree-shaped graph in Figure \ref{fig:example}.


\begin{figure}
\tikzstyle{dual}=[circle, draw, fill=black!100, inner sep=0pt, minimum width=4pt]

\begin{center}
\begin{tikzpicture}[level/.style={sibling distance=120mm/#1}]
  \tikzstyle{level 1}=[level distance=45mm]
  \tikzstyle{level 2}=[level distance=45mm]
  \tikzstyle{level 3}=[level distance=45mm]
\node [] (z){
    \begin{tikzpicture}
      \draw (0,0) node[dual] {} -- (0,1) node[dual] {} -- (0,2) node[dual] {};
      \draw (-1,0) node[dual] {} -- (0,0) node[dual] {} -- (1,0) node[dual] {};
      \draw (0,-0.3) node {\footnotesize $12$};
      \draw (-1,-0.3) node {\footnotesize $4$};
      \draw (1,-0.3) node {\footnotesize $6$};
      \draw (0.3,1) node {\footnotesize $26$};
      \draw (0.3,2) node {\footnotesize $13$};
    \end{tikzpicture}
}
child {node [] (y) {
    \begin{tikzpicture}
      \draw (0,0) node[dual] {} -- (0,1) node[dual] {} -- (0,2) node[dual] {};
      \draw (-1,0) node[dual] {} -- (0,0) node[dual] {} -- (1,0) node[dual] {};
      \draw (0,-0.3) node {\footnotesize $13$};
      \draw (-1,-0.3) node {\footnotesize $5$};
      \draw (1,-0.3) node {\footnotesize $7$};
      \draw (0.3,1) node {\footnotesize $26$};
      \draw (0.3,2) node {\footnotesize $13$};
    \end{tikzpicture}
}
  child {node [] (a) {
    \begin{tikzpicture}
      \draw (0,0) node[dual] {} -- (0,1) node[dual] {} -- (0,2) node[dual] {};
      \draw (-1,0) node[dual] {} -- (0,0) node[dual] {} -- (1,0) node[dual] {};
      \draw (0,-0.3) node {\footnotesize $2$};
      \draw (-1,-0.3) node {\footnotesize $1$};
      \draw (1,-0.3) node {\footnotesize $1$};
      \draw (0.3,1) node {\footnotesize $4$};
      \draw (0.3,2) node {\footnotesize $2$};
    \end{tikzpicture}
  }
  edge from parent[draw=none] node[left] {$2\ \ $}}
  child {node [] (l) {
    \begin{tikzpicture}
      \draw (0,0) node[dual] {} -- (0,1) node[dual] {} -- (0,2) node[dual] {};
      \draw (-1,0) node[dual] {} -- (0,0) node[dual] {} -- (1,0) node[dual] {};
      \draw (0,-0.3) node {\footnotesize $3$};
      \draw (-1,-0.3) node {\footnotesize $1$};
      \draw (1,-0.3) node {\footnotesize $2$};
      \draw (0.3,1) node {\footnotesize $6$};
      \draw (0.3,2) node {\footnotesize $3$};
    \end{tikzpicture}
  }
    child {node [] (k) {
      \begin{tikzpicture}
        \draw (0,0) node[dual] {} -- (0,1) node[dual] {} -- (0,2) node[dual] {};
        \draw (-1,0) node[dual] {} -- (0,0) node[dual] {} -- (1,0) node[dual] {};
        \draw (0,-0.3) node {\footnotesize $4$};
        \draw (-1,-0.3) node {\footnotesize $2$};
        \draw (1,-0.3) node {\footnotesize $2$};
        \draw (0.3,1) node {\footnotesize $8$};
        \draw (0.3,2) node {\footnotesize $4$};
      \end{tikzpicture}
    }
    child {node [] (kk) {
      \begin{tikzpicture}
        \draw (0,0) node[dual] {} -- (0,1) node[dual] {} -- (0,2) node[dual] {};
        \draw (-1,0) node[dual] {} -- (0,0) node[dual] {} -- (1,0) node[dual] {};
        \draw (0,-0.3) node {\footnotesize $2$};
        \draw (-1,-0.3) node {\footnotesize $1$};
        \draw (1,-0.3) node {\footnotesize $1$};
        \draw (0.3,1) node {\footnotesize $4$};
        \draw (0.3,2) node {\footnotesize $2$};
      \end{tikzpicture}
    } edge from parent[draw=none] node[left] {$2$}
  } edge from parent[draw=none] node[left] {$y$}
} edge from parent[draw=none] node[left] {$1$}
  } child {node [] (j) {
      \begin{tikzpicture}
        \draw (0,0) node[dual] {} -- (0,1) node[dual] {} -- (0,2) node[dual] {};
        \draw (-1,0) node[dual] {} -- (0,0) node[dual] {} -- (1,0) node[dual] {};
        \draw (0,-0.3) node {\footnotesize $6$};
        \draw (-1,-0.3) node {\footnotesize $2$};
        \draw (1,-0.3) node {\footnotesize $3$};
        \draw (0.3,1) node {\footnotesize $12$};
        \draw (0.3,2) node {\footnotesize $6$};
      \end{tikzpicture}
  }
    child {node [] (r) {
      \begin{tikzpicture}
        \draw (0,0) node[dual] {} -- (0,1) node[dual] {} -- (0,2) node[dual] {};
        \draw (-1,0) node[dual] {} -- (0,0) node[dual] {} -- (1,0) node[dual] {};
        \draw (0,-0.3) node {\footnotesize $6$};
        \draw (-1,-0.3) node {\footnotesize $3$};
        \draw (1,-0.3) node {\footnotesize $3$};
        \draw (0.3,1) node {\footnotesize $12$};
        \draw (0.3,2) node {\footnotesize $6$};
      \end{tikzpicture}
    }
    child {node [] (s) {
      \begin{tikzpicture}
        \draw (0,0) node[dual] {} -- (0,1) node[dual] {} -- (0,2) node[dual] {};
        \draw (-1,0) node[dual] {} -- (0,0) node[dual] {} -- (1,0) node[dual] {};
        \draw (0,-0.3) node {\footnotesize $2$};
        \draw (-1,-0.3) node {\footnotesize $1$};
        \draw (1,-0.3) node {\footnotesize $1$};
        \draw (0.3,1) node {\footnotesize $4$};
        \draw (0.3,2) node {\footnotesize $2$};
      \end{tikzpicture}
    } edge from parent[draw=none] node[left] {$3$}
  } edge from parent[draw=none] node[left] {$y^2$}
} edge from parent[draw=none] node[right] {$\ \ 1$}
} edge from parent[draw=none] node[left] {$(y^2 - x^3)^2$}};
\draw[->,dashed] (z) -- (y);
\draw[->,dashed] (z) -- (y);
\draw[->,dashed] (j) -- (r);
\draw[->,dashed] (j) -- (r);
\draw[->,dashed] (l) -- (k);
\draw[->,dashed] (l) -- (k);
\draw[->] (y) -- (l);
\draw[->] (y) -- (a);
\draw[->] (y) -- (j);
\draw[->] (r) -- (s);
\draw[->] (k) -- (kk);

\draw (-5,-7) node {$B_O$};
\draw (1,-7) node {$B_{p_1}$};
\draw (5,-7) node {$B_{p_2}$};
\draw (2,-3.5) node {$\widehat{D}^{(0)}_1 =: D^{(1)}$};
\draw (2.5,0) node {$D =: D^{(0)} = B_{p_4}$};

\draw (-1.5, -13) node {$\widehat{D}^{(1)}_1 =: D^{(2)}_1$};
\draw (4.5, -13) node {$\widehat{D}^{(1)}_2 =: D^{(2)}_2$};
\draw (-1, -17.5) node {$B_O$};
\draw (5, -17.5) node {$B_O$};

  \draw (-5,-2) node[dual] {} -- (-5,-1) node[dual] {} -- (-5,0) node[dual] {};
  \draw (-6,-2) node[dual] {} -- (-5,-2) node[dual] {} -- (-4,-2) node[dual] {};
  \draw (-5,-2.3) node {\footnotesize $p_2$};
  \draw (-6,-2.3) node {\footnotesize $O$};
  \draw (-4,-2.3) node {\footnotesize $p_1$};
  \draw (-5.3,-1) node {\footnotesize $p_4$};
  \draw (-5.3,0) node {\footnotesize $p_3$};

\end{tikzpicture}
\end{center}
\caption{Tree of divisors from Algorithm \ref{alg:A1} in Example \ref{ex:example}.} \label{fig:example}
\end{figure}
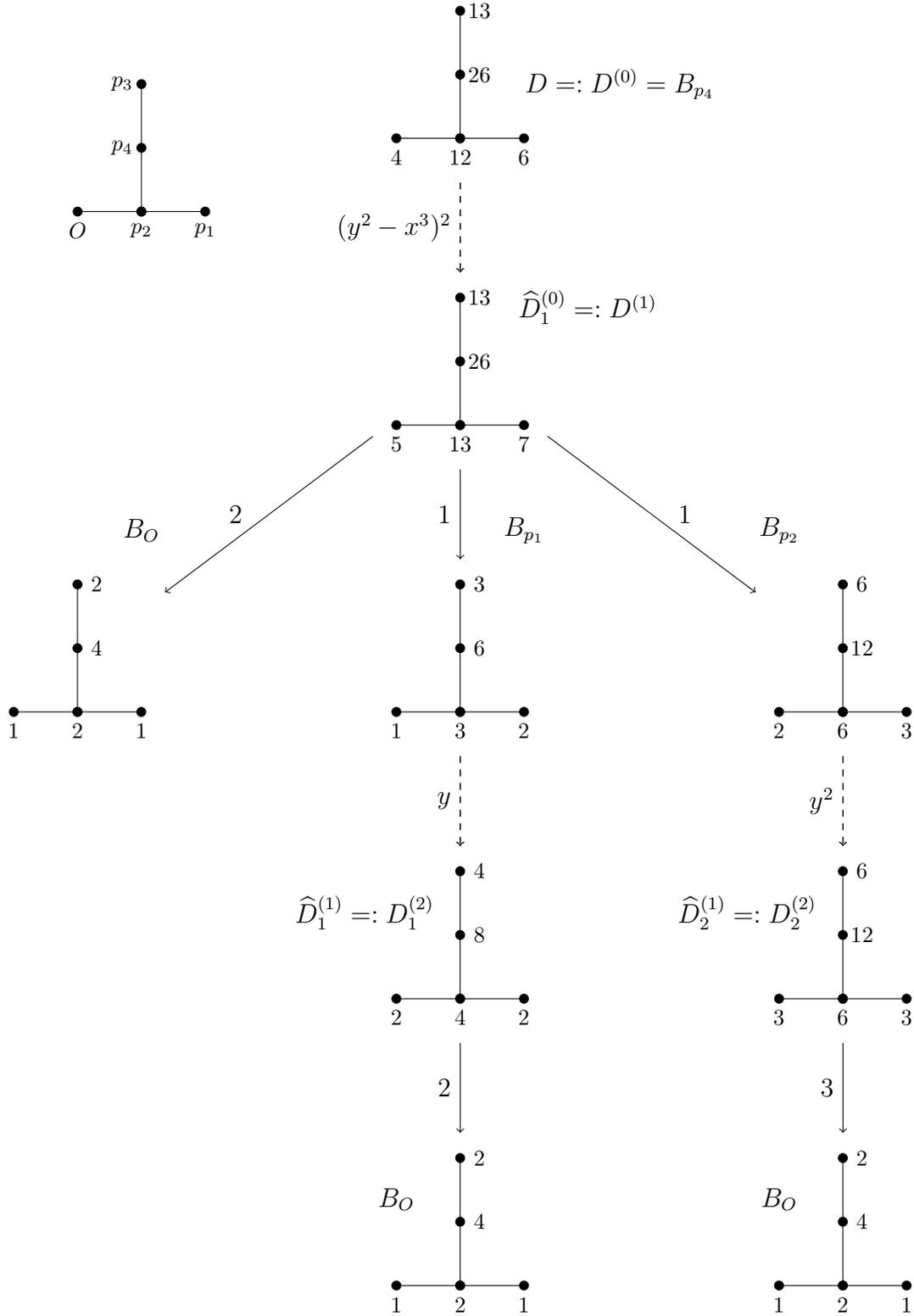

\vskip 2mm

Each vertex of the tree contains an antinef divisor. In this example, we use dual-graphs to represent them. The root node contains the initial divisor \( D \). Dashed arrows connect simple divisors $B_{q_j}$ with its corresponding adjacent $\widehat{D}_j$ from step ($i.j.2$). The maximal contact elements from step ($i.j.3$) that belong to $H_{B_{q_j}}$ but not to $H_{\widehat{D}_j}$ are represented next to dashed arrows. Solid arrows connect $\widehat{D}_j =: D^{(i+1)}$ with each of its irreducible components $B_p$, with $p \in K$. Finally, the weight $\rho^{(i)}_p$ of each divisor $B_p, p \in K,$ in $\widehat{D}^{(i)}$ is written next to the solid arrows.

\vskip 2mm

We have added some extra indices to the divisors appearing in the algorithm to highlight at which step we encounter them. Hopefully it does not create any confusion since its meaning should be clear from the context. The generators of the ideals associated to the divisors in each intermediate step are then:

\begin{itemize}
  \vskip 2mm
  \item $H_{B_O} = \mathfrak{m} = (x, y)$.
  \vskip 2mm
  \item $H_{B_{p_1}} = (y) + H_{D^{(2)}_1} = (y) + \mathfrak{m}^2 = (y, x^2, xy, y^2).$
  \vskip 2mm
  \item $H_{B_{p_2}} = (y^2) + H_{D^{(2)}_2} = (y^2) + \mathfrak{m}^3 = (y^2, x^3, x^2y, xy^2, y^3).$
  \vskip 2mm
  \item $H_{D^{(1)}} = B_{O}^2 \cdot B_{p_1} \cdot B_{p_2} = (x, y)^2 \cdot (y, x^2, xy, y^2) \cdot (y^2, x^3, x^2y, xy^2, y^3)$
  \vskip 2mm
  \hspace{0.53cm} $ = (x^7, \cancel{x^6y}, \cancel{x^5y^2}, \cancel{x^4y^3}, \dots, x^5y, x^4y^2, \dots, \cancel{x^3y^3}, \cancel{x^2y^4}, x^2y^3, xy^4, y^5)$
  \vskip 2mm
  \hspace{0.54cm} $ = (x^7, x^5y, x^4y^2, x^2y^3, xy^4, y^5)$.
  \vskip 2mm
  \item $B_{p_4} = \big((y^2 - x^3)^2\big) + H_{D^{(1)}} = \big((y^2 - x^3)^2, x^7, x^5y, x^4y^2, x^2y^3, xy^4, y^5\big)$.
  \vskip 2mm
  \item $H_{D} := H_{D^{(0)}} = B_{p_4} = \big((y^2 - x^3)^2, x^7, x^5y, x^4y^2, x^2y^3, xy^4, \cancel{y^5}\big)$.
\end{itemize}

\vskip 1.5mm

The crossed out elements are those that are redundant by step ($i.4$) and Remark \ref{remark:nakayama}. Observe that, although many crossed out elements are actually multiple of other elements, step ($i.4$) and Remark \ref{remark:nakayama} allows us to remove $y^5$ which is not multiple of any other element.

\begin{remark}
As an outcome of the algorithm, we see that \( H_D \) admits the following monomial expression
\[
H_D = \big(f_2^2, f_0^7, f_0^5f_1, f_0^4f_1^2, f_0^2f_1^3, f_0f_1^4\big),
\]
in the set of maximal contact elements $f_0=x, f_1=y, f_2=y^2-x^3$ associated to $\pi$ that we fixed in the beginning. It is worth remarking that we would get this same monomial expression for any other set of maximal contact elements chosen in the beginning.

\vskip 2mm

However, we might get a different monomial expression depending on the maximal contact elements (or powers of)  that we choose in step ($i.j.3$) of Algorithm \ref{alg:A1}. In this example, when choosing an element in $H_{B_{p_3}}$ that does not belong to $H_{\widehat{D}^{(1)}_2}$ we took $f_1^2=y^2$, but we could also had chosen $f_2=y^2 - x^3$. In this later case the final system of generators is
\[
H_D = \left( (y^2 - x^3)^2, x^2y(y^2 - x^3), xy^2(y^2 - x^3), x^7, x^5y, x^4(y^2 - x^3), x^4y^2 \right),
\]
so we get the monomial expression
\[
  H_D = \big(f_2^2, f_0^2f_1f_2, f_0f_1^2f_2, f_0^7, f_0^5f_1, f_0^4f_2, f_0^4f_1^2\big).
\]
\end{remark}
 \end{example}

\section{Correctness of the algorithm} \label{proofs}

In this section we will prove that Algorithm \ref{alg:A1} developed in Section \ref{algorithms} is correct. First, we need to check that it ends after a finite number of steps. The key point is to prove that the divisor $\widehat{D}_j$ defining the adjacent ideal to the simple ideal $H_{B_{q_j}}$ has smaller support in the total transform basis than $B_{q_j}$.

\begin{lemma} \label{lemma:smaller-support}
Using the notations in Algorithm \ref{alg:A1}, assume that $B_{q_j}$ is a simple divisor different from $B_O =  \emph{\textrm{Div}}(\pi^*\mathfrak{m})$. Let $\widehat{D}_j$ be the antinef closure of $B_{q_j} + E_O$ computed in step ($i.j.2$). Then, $\widehat{D}_j$ has smaller support than $B_{q_j}$ in the total transform basis. 
That is,
$$|\emph{\textrm{Supp}}_{\overline{E}}(B_{q_j})| > |\emph{\textrm{Supp}}_{\overline{E}}(\widehat{D}_j)|.$$
\end{lemma}

\begin{proof}
Since $B_{q_j} \ne B_O$, the excess of $B_{q_j}$ at $O$ is $\rho_O(B_{q_j}) = 0$. Then, according to Proposition \ref{prop:adjacent-ideal}, the ideal defined by $B_{q_j} + E_O$ is an adjacent ideal below $H_{B_{q_j}}$. Let $T \subseteq K$ be the points on which unloading is performed to obtain the antinef closure $\widehat{D}_j$ from $B_{q_j} + E_O$. By Proposition \ref{prop:adjacent-ideal}, $T$ are the points $p \in K$ whose associated exceptional divisor $E_p$ belongs to the same connected component as $E_O$ in $\sum_{O \le p < q_j} E_p$. Observe that $\sum_{0 \le p < q_j} E_p$ has either one or two components, according if $q_j$ is either free or satellite. In both cases, $q_j$ is proximate to the point $p \in T$ whose exceptional divisor cuts $E_{q_j}$, i.e. $E_p \cdot E_{q_j} = 1$. Hence, invoking Proposition \ref{prop:adjacent-ideal} again, the multiplicity at $q_j$ of $\widehat{D}_j$, after performing unloading on $B_{q_j} + E_O$, decreases by one. Since $B_{q_j}$ is simple, the multiplicity of $B_{q_
j}$ at $q_j$ is
one. Hence, the multiplicity of $\widehat{D}_j$ at $q_j$ is zero, giving the desired result.
\end{proof}

In the next proposition we prove that Algorithm \ref{alg:A1} ends after a finite number of steps. To emphasize the dependence of the divisors on a specific step $(i)$ of the algorithm we will use the notation $B^{(i)}_{q_j}$ and $\widehat{D}^{(i)}_j$.

\begin{proposition} \label{prop:finiteness}
Algorithm \ref{alg:A1} ends after a finite number of steps.
\end{proposition}

\begin{proof}
As noted earlier, the points \( q \in K \) such that \( e_q(B_p) \neq 0 \) are totally ordered. Then, using Equation \eqref{eq:proximity-inequality}, the sequence of multiplicities of $B_p$ decrease along those points. Hence, we have $|B_p|_{\overline{E}} = 1$ for some $p\in K$ if and only if $p=O$ and then, $B_p = \textrm{Div}(\pi^*\mathfrak{m})$.

\vskip 2mm

Using the notations in Algorithm \ref{alg:A1}, assume that we are in step $(i)$ and we have a simple divisor $B_{q_j}$ in step $(i.j.1)$.
If $q_j = O$, then $B_O = \textrm{Div}(\pi^*\mathfrak{m})$ and we are done. Otherwise, since $q_j \ne O$, we have that $|\textrm{Supp}_{\overline{E}}(B_{q_j})| > |\textrm{Supp}_{\overline{E}}(\widehat{D}^{(i)}_j)|$ by Lemma \ref{lemma:smaller-support}. Since $D^{(i+1)} := \widehat{D}^{(i)}_j$ admits a decomposition $D^{(i+1)} = \sum_{p \in K} \rho^{(i + 1)}_p B_p$, we  have $|\textrm{Supp}_{\overline{E}}(D^{(i+1)})| \ge |\textrm{Supp}_{\overline{E}}(B_p)|$ for all \( B_p \) with \( \rho{(i+1)}_p > 0 \). Hence, $|\textrm{Supp}_{\overline{E}}(B_{q_j})| > |\textrm{Supp}_{\overline{E}}(B_{p})|$, for all \( p \) with \( \rho^{(i+1)}_p > 0 \), and by induction we obtain the desired result.
\end{proof}

\begin{lemma} \label{lemma:autointersection}
Let $B_q$ be a branch basis divisor associated to a satellite point $q \in K$. Let $\Sigma(\xi) = \langle n, \check{m}_1, \dots, \check{m}_{r} \rangle$ be the semigroup of $\xi: f_q = 0$ and take $\gamma : f_r = 0$ such that $[\gamma . \xi]_O = \check{m}_r$. Then, $B_q^2 = [\zeta . \xi]_O$ with $\zeta : f^{n_{r-1}}_r = 0$. Furthermore, $v_p(f_r^{n_{r-1}}) \ge v_p(B_q)$, for all $p \le q$.
\end{lemma}

\begin{proof}
Assume that $B_q = \sum_{p \in K} e_p \overline{E}_p$. The first claim follows from the following computation:
\[
  B_q^2 = \sum_{p \in K} e_p^2 = \sum_{i = 1}^r (m_i - m_{i-1}) n_{i-1} = \sum_{i = 1}^{r-1} (n_{i-1} - n_i)m_i + n_{r-1}m_r = n_{r-1}\check{m}_r =  n_{r-1} [\gamma . \xi]_O=   [\zeta . \xi ]_O.
\]
where the second equality is true since  $q \in K$ is a satellite point (see  \cite[\S 5.10]{casas}) and the fourth equality comes from \eqref{eq:semigroup-generators}.

\vskip 2mm

To prove the second claim, in virtue of Remark \ref{remark:values_as_intersection}, it suffices to check the inequalities $[\zeta . \eta_p]_O \ge  [\xi . \eta_p]_O$ for any $p \le q$ with $\eta_p: f_p = 0$. We will use known properties of the ultrametric $d_{\cC}$ distance (introduced in \cite{ploski}) defined over the space $\cC$ of plane branches as
$ \frac{1}{d_\cC(C,D)} =\frac{[C . D]_O}{e_O (C) e_O(D)}$ for any $C,D \in \cC$. Hence, the inequalities above are equivalent to  $d_{\cC}(\xi,\eta_p)\ge d_{\cC}(\gamma,\eta_p)$ for any  $p \le q$, since in our case
\[
  \frac{[\zeta.\eta_p]_O}{e_O (\zeta) e_O(\eta_p)}= \frac{n_{g-1} [\gamma.\eta_p]_O}{n_{g-1} e_O (\gamma) e_O(\eta_p)} =\frac{1}{d_{\cC}(\gamma,\eta_p)} .
\]

Notice that $f_q =f_{q_g}$ for some point $q_g$, $O \leq q_g \leq q$, which corresponds to a dead-end in the dual graph of $B_q$. Now we summarize the results on the ultrametric space of plane branches of \cite[3.1, 3.2 and 3.4]{AACGA11} adapted to our setting:

\begin{itemize}
  \item[$\cdot$] $d_{\cC}(\xi,\eta_p)= d_{\cC}(\gamma,\eta_p)$, if $O \leq p < q_g$;

  \item[$\cdot$] $d_{\cC}(\xi,\eta_p)= d_{\cC}(\gamma,\eta_p)$, if $q_g < p \leq q$ and in the dual graph of $B_q$ the vertex of $p$ lies on the segment joining the vertexes $q$ and $q_g$;

  \item[$\cdot$] $d_{\cC}(\xi,\eta_p)> d_{\cC}(\gamma,\eta_p)$, otherwise.
\end{itemize}
Hence, the second claim follows.
\end{proof}

\begin{proposition} \label{prop:good-curve}
Using the notations in Algorithm \ref{alg:A1}, at any step $(i)$ of the algorithm, there exists a power of a maximal contact element $f^{n_j}_{\tau_j} \in \mathcal{O}_{X, O}$ as required at step $(i.j.3)$ and such element belongs to $H_{B_{q_j}}$ but not to $H_{\widehat{D}_j}$.
\end{proposition}

\begin{proof}
We are going to break the proof of the first statement in two cases depending on whether the point $q_j \in K$ is free or satellite. With the notations from  step $(i.j.3)$, $p \in K$ will be the last free point such that $e_p(B_{q_j}) \neq 0$.

\vskip 2mm

Assume first that $q_j$ is free, i.e. $p = q_j$. If, in addition, the vertex of $q_j$ is a dead-end of the dual graph of $F$ we are done, since $f_{\tau_j} = f_{q_j}$ and $f_{q_j} \in H_{B_{q_j}}$.
If $q_j$ is not a dead-end of the dual graph, there is a dead-end $q \in K$ and a totally ordered sequence $q_j \le p_1 \le \cdots \le p_r \le q$ of free points such that $e_{q_j}(B_q) = e_{p_1}(B_q) = \cdots = e_{p_r}(B_q) = e_q(B_q) = 1$, by \eqref{eq:proximity-inequality}. Therefore, $f_{\tau_j} = f_q$ with $e_{q_j}(f_q) = 1$ and $B_q = B_p + \overline{E}_{p_1} + \cdots + \overline{E}_{p_r} + \overline{E}_q$, which implies, by \eqref{eq:proximity}, that $f_q \in H_{B_q} \varsubsetneq H_{B_{q_j}}$. In either case we have that $n_j = e_O(B_{q_j})/e_O(f_{\tau_j}) = 1$.

\vskip 2mm

Now, assume that $q_j$ is satellite and hence $p < q_j$. Let $\Sigma(\xi) = \langle n, \check{m}_1, \dots, \check{m}_g\rangle$ the semigroup of $\xi : f_{q_j} = 0$. By \cite[\S 5.8]{casas}, $p$ has the property that any $\gamma : f_p = 0, f_p \in \mathcal{O}_{X, O}$ satisfies $[\gamma . \xi]_O = B_p \cdot B_{q_j} = \check{m}_g$. However,  it may happen that $p \in K$ is not a dead-end. In this case, using the same argument as before, there is a dead-end $q \in K$ and a totally ordered sequence $p \le p_1 \le \cdots \le p_r \le q$ of free points and $B_q = B_p + \overline{E}_{p_1} + \cdots +\overline{E}_{p_r} +  \overline{E}_q$. Since $p$ is the last free point of $B_{q_j}$, $B_{q_j} \cdot \overline{E}_{p_i} = 0$ for $i = 1, \dots, r$ and also $B_{q_j} \cdot \overline{E}_{q} = 0$. Hence, $[\widetilde{\gamma} . \xi]_O = B_q \cdot B_{q_j} =  \check{m}_g$ with $\widetilde{\gamma} : f_q = 0$, i.e. we can take $f_{\tau_j} = f_q$. We can then apply Lemma \ref{lemma:autointersection} to $B_{q_j}$ with $f_g = f_{\tau_
j}
$ yielding that $f_{\tau_j}^{n_j} \in H_{B_{q_j}}$ with $n_j = e_O(B_{q_j})/e_O(f_{\tau_j})$.

\vskip 2mm

Finally, if $f_{\tau_j}^{n_j}\in \mathcal{O}_{X, O}$ fulfills the requirements of step $(i.j.3)$, then  $e_O(f_{\tau_j}^{n_j}) = e_O(B_{q_j})$, but $e_O(\widehat{D}_j) > e_O(B_{q_j})$ by Proposition \ref{prop:adjacent-ideal}, therefore we have that $f_\tau^{n_j} \not\in H_{\widehat{D}_j}$.
\end{proof}


\begin{theorem} \label{thm:main}
Let \( \pi : X' \rightarrow X \) be a proper birational morphism and let \( D \in \Lambda_{\pi} \). Then, Algorithm \ref{alg:A1} computes a set of generators for \( H_D \) that are monomial in any given set of maximal contact elements of \( \pi \).
\end{theorem}

\begin{proof}
Let us prove that the $i$-th step of the algorithm returns a system of generators of $H_{D^{(i)}}$ which has the desired properties. By Zariski's factorization Theorem \ref{thm:zariski}, it is enough to focus on computing generators for each simple ideal $H_{B_{q_j}}, j = 1, \dots, d_i,$ in the decomposition of $D^{(i)}$. Fixing $B_{q_j}$ at step $(i.2)$, we will make induction on the order of the neighbourhood that $q_j \in K$ belongs to, and we will show that Algorithm  \ref{alg:A1} computes generators for $H_{B_{q_j}}$ which are monomials in the set of maximal contact elements.

\vskip 2mm

If $q_j = O$, then $B_O = \textrm{Div}(\pi^*\mathfrak{m})$ and step ($i.j$.1) returns $H_{D^{(i)}} = \mathfrak{m}$, since a pair of smooth transverse elements generate $\mathfrak{m}$. By construction, any set of maximal contact elements contain such a pair of elements.

\vskip 2mm

Assume now that $q_j \neq O$ and that the algorithm  computes the generators of the ideals associated to $B_{p}$ for $ p < q_j$. By Proposition \ref{prop:adjacent-ideal}, $H_{\widehat{D}_j} \varsubsetneq H_{B_{q_j}}$ are adjacent ideals. Since $\widehat{D}_j = \sum_{p < q_j} \rho^{(i)}_p B_p$, we can apply the induction hypothesis to the simple divisors $B_{p}, p < q_j$ such that $\rho^{(i)}_p \ne 0$ and apply Theorem \ref{thm:zariski} to get
\[
  H_{\widehat{D}_j} = \prod_{p < q_j} H_{B_p}^{\rho^{(i)}_p} \varsubsetneq H_{D^{(i)}}.
\]
At this point is it enough to add any element that belongs to $H_{B_{q_j}}$ but not to $H_{\widehat{D}_j}$ to get a system of generators of $H_{B_{q_j}}$. By Proposition \ref{prop:good-curve}, the element chosen at step ($i.j$.3) has the desired properties, namely, it is a power of a maximal contact element. Finally, we can remove unnecessary elements from the system of generators of $H_{D^{(i)}}$ using Lemma \ref{lemma:nakayama}.

\vskip 2mm

The dependency on the set of maximal contact elements \( \{f_i\}_{i\in I} \) is only used in step $(i.j.3)$. The conditions required to \( \{f_p\}_{i\in I} \) depend only on a finite number of valuations associated to the exceptional divisors of \(\pi\). These conditions are fulfilled by an infinite number of elements which can be part of a set of maximal contact elements.
\end{proof}

\begin{remark}
We would like to stress the generality of the monomial generators in Theorem \ref{thm:main}. Consider the monomials \(\mathbf{z}^{\boldsymbol{\alpha}} = \prod_{i \in I} z_i^{\alpha_i}, \boldsymbol{\alpha} = (\alpha_i)_{i \in I} \) in the variables \( z_i, \hskip 2mm i \in I \). Each variable \( z_i \) formally represents all possible elements \( f_p \) for a fixed dead-end of the dual graph of \( \pi \). They all have the same value for the valuations associated to the exceptional divisors. Take now any set of maximal contact elements \( \boldsymbol{f} = \{f_{p_i}\}_{i \in I} \) and denote \( \boldsymbol{z}^{\boldsymbol{\alpha}}_{\boldsymbol{f}} = \prod_{i \in I} f_{p_i}^{\alpha_i} \) the specialization \( z_i \mapsto f_{p_i} \).

The result of Theorem \ref{thm:main} is that Algorithm \ref{alg:A1} returns formally \( (\boldsymbol{z}^{\boldsymbol{\alpha}_1}, \dots, \boldsymbol{z}^{\boldsymbol{\alpha}_r}) \) and that for any two sets of maximal contact elements \( \boldsymbol{f} = \{f_{p_i}\}_{i \in I} \) and \( \boldsymbol{g} = \{g_{p_i}\}_{i \in I} \), both specializations \( H_{D, \boldsymbol{f}}  = (\boldsymbol{z}_{\boldsymbol{f}}^{\boldsymbol{\alpha}_1}, \dots, \boldsymbol{z}_{\boldsymbol{f}}^{\boldsymbol{\alpha}_r}) \) and \(H_{D, \boldsymbol{g}} = (\boldsymbol{z}_{\boldsymbol{g}}^{\boldsymbol{\alpha}_1}, \dots, \boldsymbol{z}_{\boldsymbol{g}}^{\boldsymbol{\alpha}_r}) \) are equal, \( H_D = H_{D, \boldsymbol{f}} = H_{D, \boldsymbol{g}} \), and are wanted the complete ideal $H_D$.
\end{remark}

\begin{corollary}\label{cor:monomial_exp}
The set of monomial expressions returned by Algorithm \ref{alg:A1} is a topological invariant of $D$, i.e. it is an invariant of the weighted dual graph of $D$.
\end{corollary}
\begin{proof}
From the proof of Theorem \ref{thm:main}
it follows that the algorithm only uses the information of the equisingularity types of the maximal contact elements and of the dual graph of $D$ weighted by the natural numbers $v_p(D)$ for $p \in K$.
\end{proof}



\section{Application to some families of complete ideals}

Let $(X,O)$ be a germ of smooth complex surface and let \( \pi : X' \longrightarrow {X} \) be a proper birational transform. The sheaf ideals \( O_{X'}(D) \) and its pushforward, for some \( D \in \Lambda_{\pi} \), arise in many different contexts. The goal of this sections is two show how Algorithm \ref{alg:A1} and Theorem \ref{thm:main} applies to different problems.

\vskip 2mm

Our approach is specially useful when studying families of divisor \( \{D_i\} \) in \( \Lambda_{\pi} \) since all the generators for all the ideals \( H_{D_i} \) will be given as monomials in any set of maximal contact elements. 

\subsection{Integral closure}

\vskip 2mm


Let \( \mathfrak{a} \subseteq \mathcal{O}_{X, O}\) be an ideal which can be assumed to be \( \mathfrak{m} \)-primary after considering the decomposition \( \mathfrak{a} = (a) \cdot \mathfrak{a}' \) with \( a = \gcd(\mathfrak{a}) \). Let $\pi: X' \rightarrow X$ be a log-resolution of the ideal $\fa$, i.e., a proper birational morphism such that there exists $F$ an effective Cartier divisor such that $\fa\cdot\cO_{X'} = \cO_{X'}\left(-F\right)$. Then, the integral closure $\overline{\fa}$ of \( \mathfrak{a} \) is just the ideal $H_F$, 
see \cite[\S 8.3]{casas}.

\vskip 2mm
Therefore, we have a very simple method to compute the integral closure of any planar ideal that boils down to the following steps:

\vskip 2mm
{\it
$\cdot$ Compute the divisor $F$ associated to the log-resolution of \( \mathfrak{a} \) by  using the algorithm from \cite{ACAMB16}.

\vskip 2mm

$\cdot$ Compute a set of generators for the ideal $H_F$ using Algorithm \ref{alg:A1}.}


\vskip 2mm

Let us illustrate this situation with an small example.

\begin{example}
Let $\mathfrak{a} = ((y^2-x^3)^2,x^2y^3) \subseteq \mathcal{O}_{X, O}$ an ideal. One can compute the log-resolution divisor of \( \mathfrak{a} \) using the algorithm from \cite{ACAMB16}. The log-resolution and its associated divisor \( F \) of $\mathfrak{a}$ are precisely the proper birational morphism \( \pi \) and the divisor \( D \) from Example \ref{ex:example1}.

\vskip 2mm

Namely, \( F = 4E_O + 6E_{p_1} + 12E_{p_2} + 13E_{p_3} + 26E_{p_4} \) and we can take \( f_0 = x, f_1 = y, f_2 = y^2 - x^3 \) as a set of maximal contact elements of the log-resolution \( \pi \). Thus, from the computation in Example \ref{ex:example}, one deduces that
\[ \overline{\mathfrak{a}} = H_F = \pi_* \mathcal{O}_{X'}(F) = ((y^2 - x^3)^2, x^7, x^5y, x^4y^2, x^2y^3, xy^4).\]
\end{example}

The following results follows directly from Theorem \ref{thm:main} and its corollaries.

\begin{theorem} \label{thm:integral-closure}
Let \( \mathfrak{a} \subseteq \mathcal{O}_{X, O} \) be an ideal. There exists a set of generators of the integral closure \( \mathfrak{a} \) that are monomial in any given set of maximal contact elements of the log-resolution of \( \mathfrak{a} \).
\end{theorem}

\begin{corollary} \label{cor:monomial}
Let $\pi: X' \lra X$ be a proper birational morphism. Any complete ideal $\fa \subseteq \cO_{X,O}$ whose log-resolution is dominated by $\pi$ admits a system of generators given by monomials in any set of maximal contact elements associated to $\pi$.
\end{corollary}

\begin{corollary}\label{cor:monomial_exp2}
The set of monomial expressions returned by Algorithm \ref{alg:A1} for the integral closure of an ideal $\fa$  is an equisingular invariant of $\fa$.
\end{corollary}

\subsection{Multiplier ideals}

Let \( \pi : X' \longrightarrow X \) be a log-resolution of an ideal \( \mathfrak{a} \subseteq \mathcal{O}_X \) and
let \( F \) be the divisor such that \( \mathfrak{a} \cdot \mathcal{O}_X' = \mathcal{O}_{X'}(-F) \).
The \emph{multiplier ideal} associated to \( \mathfrak{a} \) and some rational
number \( \lambda \in \mathbb{Q}_{\geq 0} \) is defined as
\[ \mathcal{J}(\mathfrak{a}^{\lambda}) = \pi_* \mathcal{O}_{X'}(\lceil K_\pi - \lambda F\rceil), \]
where \( K_\pi \) is the so-called \emph{relative canonical divisor} (see \cite{Laz04} for details).
Multiplier ideals form a discrete nested sequence of ideals
\[ \mathcal{O}_X \supsetneq \mathcal{J}(\mathfrak{a}^{\lambda_0}) \supsetneq \mathcal{J}(\mathfrak{a}^{\lambda_{1}}) \supsetneq \mathcal{J}(\mathfrak{a}^{\lambda_{2}}) \supsetneq \cdots \supsetneq \mathcal{J}(\mathfrak{a}^{\lambda_{i}}) \supsetneq \cdots \]
and the rational numbers $0 < \lambda_0 < \lambda_1 < \cdots $ where an strict inclusion of ideals is achived are
called the \emph{jumping numbers} associated to \( \mathfrak{a} \).

\vskip 2mm

There are general algorithms as those developed by Shibuta \cite{Shi11} and  Berkersch and Leykin \cite{BL10} that,
given a set of generators of $ \mathfrak{a}$, return the list of jumping numbers and a minimal set of generators
of the corresponding multiplier ideals. These algorithms use the theory of Bernstein-Sato polynomials
and have been implemented in {\tt Macaulay2}. However, it is difficult to compute large examples due to the complexity of
these algorithms.

\vskip 2mm

In the case of planar ideals, there are methods given by
J\"arviletho \cite{Jar11}, Naie \cite{Nai09} and Tucker \cite{Tuc10} to compute the set of jumping numbers.
The first two authors of this manuscript and Dachs-Cadefau \cite{ACAMDC} gave an algorithm that computes
sequentially the list of jumping numbers of a planar ideal and the antinef divisor associated to the
corresponding multiplier ideal.

\vskip 2mm

Combining the algorithms developed in \cite{ACAMB16} and \cite{ACAMDC} with
 Algorithm \ref{alg:A1} we may  provide a method that, given a set of generators of a planar ideal $ \mathfrak{a}$, returns the set of jumping numbers and a set of generators of the corresponding multiplier ideals. Namely, we have to perform the following steps:

\vskip 2mm
{\it
$\cdot$ Compute the divisor $F$ associated to the log-resolution of \( \mathfrak{a} \) by  using the algorithm from \cite{ACAMB16}.

\vskip 2mm

$\cdot$  Compute the sequence of jumping numbers $\{ \lambda_j \}_{j\in \bZ_{\geq 0}}$ and the divisor corresponding to the associated
multiplier ideals $\{  \mathcal{J}(\mathfrak{a}^{\lambda_j})\}_{j\in \bZ_{\geq 0}}$ , i.e. the antinef closures $D_{\lambda_j}$ of $\lfloor \lambda_j F - K_{\pi} \rfloor$,  using the main algorithm of \cite{ACAMDC}.

\vskip 2mm

$\cdot$ Compute a set of generators for the ideals $H_{D_{\lambda_j}}$ using Algorithm \ref{alg:A1}.}

\vskip 2mm

This method is illustrated with the following:

\begin{example}
Consider the ideal $\mathfrak{a}=((y^2 - x^3)^3, x^3(y^2 - x^3)^2, x^6y^3) \subseteq \mathcal{O}_{X, O}$.
The log-resolution of $\mathfrak{a}$ can be computed using the algorithm from \cite{ACAMB16} and is represented by means of the following dual-graph:

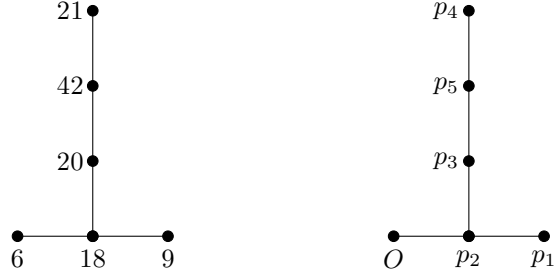
\tikzstyle{dual}=[circle, draw, fill=black!100, inner sep=0pt, minimum width=4pt]
\begin{center}
\begin{tikzpicture}[level/.style={sibling distance=120mm/#1}]
  \draw (-5,-2) node[dual] {} -- (-5,-1) node[dual] {} -- (-5,0) node[dual] {} -- (-5,1) node[dual] {};
  \draw (-6,-2) node[dual] {} -- (-5,-2) node[dual] {} -- (-4,-2) node[dual] {};
  \draw (-5,-2.3) node {\footnotesize $p_2$};
  \draw (-6,-2.3) node {\footnotesize $O$};
  \draw (-4,-2.3) node {\footnotesize $p_1$};
  \draw (-5.3,-1) node {\footnotesize $p_3$};
  \draw (-5.3,0) node {\footnotesize $p_5$};
  \draw (-5.3,1) node {\footnotesize $p_4$};

  \draw (-10,-2) node[dual] {} -- (-10,-1) node[dual] {} -- (-10,0) node[dual] {} -- (-10,1) node[dual] {};
  \draw (-11,-2) node[dual] {} -- (-10,-2) node[dual] {} -- (-9,-2) node[dual] {};
  \draw (-10,-2.3) node {\footnotesize $18$};
  \draw (-11,-2.3) node {\footnotesize $6$};
  \draw (-9,-2.3) node {\footnotesize $9$};
  \draw (-10.3,-1) node {\footnotesize $20$};
  \draw (-10.3,0) node {\footnotesize $42$};
  \draw (-10.3,1) node {\footnotesize $21$};
\end{tikzpicture}
\end{center}

The divisor  \( F \) such that \( \mathfrak{a} \cdot \mathcal{O}_X' = \mathcal{O}_{X'}(-F) \) is
$F = 6E_O + 9E_1 + 18E_2 + 20E_3 + 21E_4 + 42E_5$.
A set of maximal contact elements for the log-resolution of $\mathfrak{a}$ is, for instance,
$f_0 = x, f_1 = y, f_2 = y^2 - x^3$.

\vskip 2mm

The jumping numbers smaller than 1 computed using the algorithm from \cite{ACAMDC} and generators for the associated multiplier ideals computed using Algorithm \ref{alg:A1} can be found in Table 1. 
\end{example}

\def\arraystretch{1.3}
\begin{table}[!ht] 
\centering
\begin{tabular}{|c|c|}
\hline
\(\lambda_{i}\) & \(\mathcal{J}(\mathfrak{a}^{\lambda_i})\) \\[1.5pt] \hline
\hline
\(\frac{5}{18}\)  &  \( x, y\) \\[1.5pt] \hline
\(\frac{7}{18}\)  &  \( y, x^2\) \\[1.5pt] \hline
\(\frac{4}{9}\)  &  \( x^2, xy, y^2\) \\[1.5pt] \hline
\(\frac{1}{2}\)  &  \( xy, y^2, x^3\) \\[1.5pt] \hline
\(\frac{23}{42}\)  &  \( y^2, x^3, x^2y\) \\[1.5pt] \hline
\(\frac{25}{42}\)  &  \( y^2 - x^3, x^2y, xy^2, x^4\) \\[1.5pt] \hline
\(\frac{11}{18}\)  &  \( x^2y, xy^2, y^3, x^4\) \\[1.5pt] \hline
\(\frac{9}{14}\)  &  \( xy^2, y^3, x^4, x^3y\) \\[1.5pt] \hline
\(\frac{29}{42}\)  &  \( x(y^2 - x^3), y(y^2 - x^3), x^3y, x^2y^2, x^5\) \\[1.5pt] \hline
\(\frac{13}{18}\)  &  \( y^3, x^3y, x^2y^2, x^5\) \\[1.5pt] \hline
\(\frac{31}{42}\)  &  \( y(y^2 - x^3), x^2y^2, xy^3, x^2(y^2 - x^3), x^4y\) \\[1.5pt] \hline
\(\frac{7}{9}\)  &  \( x^2y^2, xy^3, y^4, x^5, x^4y\) \\[1.5pt] \hline
\(\frac{11}{14}\)  &  \( x^2(y^2 - x^3), xy(y^2 - x^3), y^2(y^2 - x^3), x^4y, x^3y^2, x^6\) \\[1.5pt] \hline
\(\frac{5}{6}\)  &  \( xy(y^2 - x^3), y^2(y^2 - x^3), x^3(y^2 - x^3), x^3y^2, x^2y^3, x^5y \) \\[1.5pt] \hline
\(\frac{37}{42}\)  &  \( xy(y^2 - x^3), y^2(y^2 - x^3), x^3(y^2 - x^3), x^2y^3, xy^4, x^5y, x^4y^2, x^7\) \\[1.5pt] \hline
\(\frac{8}{9}\)  &  \( y^2(y^2 - x^3), x^5y, x^3(y^2 - x^3), x^2y^3, x^2y(y^2 - x^3), xy^4, x^4y^2, x^7\) \\[1.5pt] \hline
\(\frac{13}{14}\)  &  \( y^2(y^2 - x^3), x^3(y^2 - x^3), x^2y(y^2 - x^3), xy^4, y^5, x^4y^2, x^3y^3, x^7, x^6y\) \\[1.5pt] \hline
\(\frac{17}{18}\)  &  \( x^2y(y^2 - x^3), xy^2(y^2 - x^3), y^3(y^2 - x^3), x^4y^2, x^4(y^2 - x^3), x^3y^3, x^6y\) \\[1.5pt] \hline
\(\frac{41}{42}\)  &  \( x^2y(y^2 - x^3), xy^2(y^2 - x^3), y^3(y^2 - x^3), x^4(y^2 - x^3), x^3y^3, x^2y^4, x^6y, x^5y^2, x^8\) \\[1.5pt] \hline
\end{tabular}
\vspace{5pt}
\caption{The jumping numbers smaller than 1 and generators of the associated multiplier ideal for $\mathfrak{a}=((y^2 - x^3)^3, x^3(y^2 - x^3)^2, x^6y^3)$.} \label{table:T1}
\end{table}

It is a known result that the multiplier ideals \( \mathcal{J}(\mathfrak{a}^{\lambda}) \) associated to \( \mathfrak{a} \) are the same when taking the completion \( \overline{\mathfrak{a}} \) of \( \mathfrak{a} \), i.e. \( \mathcal{J}(\mathfrak{a}^{\lambda}) = \mathcal{J}(\overline{\mathfrak{a}}^{\lambda}) \). As a corollary of Theorem \ref{thm:main} and Corollary \ref{cor:monomial} we obtain the following result which resembles Howald's theorem \cite{howald} on the fact that multiplier ideals of monomial ideals are also monomial.

\begin{theorem}
Let \( \mathfrak{a} \subseteq \mathcal{O}_{X, O} \) be an ideal and consider  its completion \( \overline{\mathfrak{a}} \), that can be generated by monomials in any given set of maximal contact elements. Then, the multiplier ideals \( \mathcal{J}(\mathfrak{a}^{\lambda}) \) are also generated by  monomials in the same set of maximal contact elements of the log-resolution of \( \mathfrak{a} \).
\end{theorem}

\subsection{Valuation filtration}
Throughout this subsection we will use the notation from Section \ref{max_contact}. Let  $\xi : f = 0, f \in \mathcal{O}_{X, O}$
 be a plane branch and let $v_\xi$ be the valuation induced by the intersection multiplicity.
Let  $\mathfrak{V}_i$ denote the ideal of all the elements in $\mathcal{O}_{X, O}$ with valuation greater or equal to $i$:
\[
\mathfrak{V}_i := \{ \zeta \in \mathcal{O}_{X, O} \ |\ v_\xi(\zeta) \geq i\}.
\]
These ideals form a filtration in \( \mathcal{O}_{X, O} \):
\begin{equation} \label{eq:filtration-valuation}
\mathcal{O}_{X,O} = \mathfrak{V}_0 \supseteq \mathfrak{V}_1 \supseteq \cdots \supseteq \mathfrak{V}_i \supseteq \mathfrak{V}_{i+1} \supseteq \cdots
\end{equation}
such that \( \mathfrak{V}_i \cdot \mathfrak{V}_j \subset \mathfrak{V}_{i+j} \) and \( \cap_{i \in \mathbb{Z}_{\geq 0}} \mathfrak{V}_i = (f) \). This type of filtration was considered by Teissier in \cite{Teissier}. Since, the ideals \( \mathfrak{V}_i \) are defined by valuations, they are complete, and hence, have the form \( \pi_* \mathcal{O}_{X_i'}(D_i) \) for some effective divisor \( D_i \) in some surface \( X_i' \).

\vskip 2mm

\begin{example}

Consider the germ of plane curve $\xi: f=0$,  \( f = (y^2 - x^3)^2 - x^5y \in \mathcal{O}_{X, O} \).
A resolution morphism \( \pi : X' \longrightarrow X \) of \( \xi \) is the same as the one given
for the ideal in Example \ref{ex:example1}. Using Algorithm \ref{alg:A1} we obtain the
generators of the filtration \eqref{eq:filtration-valuation} until the last $\mathfrak{V}_i$ dominated by the log-resolution of $(f)$, which we collect in Table 2. 
\end{example}

\def\arraystretch{1.3}
\begin{table}[!ht] \label{table:T2}
\centering
\begin{tabular}{|c|c|}
\hline
\(i\) & \(\mathfrak{V}_i\) \\[1.5pt] \hline
\hline
\(1,2,3,4\) & \(x,y\) \\[1.5pt] \hline
\(5, 6\) & \(y,x^2\) \\[1.5pt] \hline
\(7, 8\) & \(xy,x^2,y^2\) \\[1.5pt] \hline
\(9, 10\) & \(xy,y^2,x^3\) \\[1.5pt] \hline
\(11, 12\) & \(x^2y,y^2,x^3\) \\[1.5pt] \hline
\(13\) & \(y^2 - x^3,x^2y,xy^2,x^4\) \\[1.5pt] \hline
\(14\) & \(x^2y,xy^2,y^3,x^4\) \\[1.5pt] \hline
\(15, 16\) & \(xy^2,y^3,x^4,x^3y\) \\[1.5pt] \hline
\(17\) & \(x(y^2 - x^3),y(y^2 - x^3),x^3y,x^2y^2,x^5\) \\[1.5pt] \hline
\(18\) & \(y^3,x^3y,x^2y^2,x^5\) \\[1.5pt] \hline
\(19\) & \(y(y^2 - x^3),x^2(y^2 - x^3),x^4y,x^2y^2,xy^3\) \\[1.5pt] \hline
\(20\) & \(x^4y,x^2y^2,xy^3,y^4,x^5\) \\[1.5pt] \hline
\(21\) & \(x^2(y^2 - x^3),xy(y^2 - x^3),y^2(y^2 - x^3),x^4y,x^3y^2,x^6\) \\[1.5pt] \hline
\(22\) & \(xy^3,y^4,x^4y,x^3y^2,x^6\) \\[1.5pt] \hline
\(23\) & \(xy(y^2 - x^3),y^2(y^2 - x^3),x^3y^2,x^3(y^2 - x^3),x^2y^3,x^5y\) \\[1.5pt] \hline
\(24\) & \(y^4,x^3y^2,x^2y^3,x^6,x^5y\) \\[1.5pt] \hline
\(25\) & \((y^2 - x^3)^2,x^3(y^2 - x^3),x^2y(y^2 - x^3),x^5y,x^4y^2,x^7\) \\[1.5pt] \hline
\(26\) & \((y^2 - x^3)^2,x^2y^3,xy^4,x^5y,x^4y^2,x^7
\) \\[1.5pt] \hline
\end{tabular}
\vspace{5pt}
\caption{The ideals \(\mathfrak{V}_i\) of the filtration associated to the plane branch \( f = (y^2 - x^3)^2 - x^5y \) for \( i = 1, 2, \dots, 26 \).}
\end{table}

\end{document}